\documentclass[a4paper,11pt,reqno]{amsart}
\usepackage[utf8]{inputenc}
\usepackage{graphicx}
\usepackage{a4wide}
\usepackage[english]{babel}
\usepackage[T1]{fontenc}
\usepackage{amsmath,amssymb,amsthm,stmaryrd}
\usepackage{mathrsfs}
\usepackage{esint}
\usepackage{mathtools}
\mathtoolsset{showonlyrefs}
\usepackage{comment}
\usepackage{color}
\usepackage{extarrows}
\usepackage{tikz}
\usepackage{tcolorbox}
\usetikzlibrary{arrows}

\newcommand{\R}{\mathbb{R}}

\newcommand{\Sb}{\mathbb{S}}

\newcommand{\N}{\mathbb{N}}
\newcommand{\Z}{\mathbb{Z}}
\newcommand{\Ds}{\mathscr{D}}

\newcommand{\Sc}{\mathcal{S}}

\newcommand{\Ms}{\mathscr{M}}
\newcommand{\Ns}{\mathscr{N}}

\newcommand{\Qc}{\mathcal{Q}}

\newcommand{\diam}{\text{diam}}
\newcommand{\dist}{\text{dist}}

\newcommand{\eps}{\varepsilon}

\newcommand{\bu}{\bullet}

\def\Xint#1{\mathchoice
  {\XXint\displaystyle\textstyle{#1}}%
  {\XXint\textstyle\scriptstyle{#1}}%
  {\XXint\scriptstyle\scriptscriptstyle{#1}}%
  {\XXint\scriptscriptstyle\scriptscriptstyle{#1}}%
  \!\int}
\def\XXint#1#2#3{{\setbox0=\hbox{$#1{#2#3}{\int}$}
    \vcenter{\hbox{$#2#3$}}\kern-.5\wd0}}

\def\avgint{\Xint-}
\newcommand{\vertiii}[1]{{\left\vert\kern-0.25ex\left\vert\kern-0.25ex\left\vert #1
    \right\vert\kern-0.25ex\right\vert\kern-0.25ex\right\vert}}

\newtheorem{theorem}[equation]{Theorem}
\newtheorem*{theorem*}{Theorem}
\newtheorem{lemma}[equation]{Lemma}
\newtheorem{openproblem}[equation]{Open problem}
\newtheorem{corollary}[equation]{Corollary}
\newtheorem{proposition}[equation]{Proposition}

\theoremstyle{definition}
\newtheorem{defin}[equation]{Definition}
\newtheorem*{defin*}{Definition}

\newtheorem*{question*}{Question}
\newtheorem{example}[equation]{Example}
\newtheorem{remark}[equation]{Remark}
\newtheorem*{remark*}{Remark}

\numberwithin{equation}{section}

\usepackage{rotating}

\makeatletter
\newcommand\@notni[2]{\mathrel{\rotatebox[y=#1]{180}{$#2\notin$}}}
\newcommand\notni{
\mathchoice
  {\@notni{0.57ex}\displaystyle}
  {\@notni{0.57ex}\textstyle}
  {\@notni{0.39ex}\scriptstyle}
  {\@notni{0.26ex}\scriptscriptstyle}
}
\makeatother

\title{$C_p$ estimates for rough homogeneous singular integrals and sparse forms}

\author{Javier Canto, Kangwei Li, Luz Roncal and Olli Tapiola}

\address[Javier Canto]{BCAM - Basque Center for Applied Mathematics, 48009 Bilbao, Spain}
\email{jcanto@bcamath.org}

\address[Kangwei Li]{Center for Applied Mathematics, Tianjin University, Weijin Road 92, 300072 Tianjin, China}
\email{kangwei.nku@gmail.com}

\address[Luz Roncal]{BCAM - Basque Center for Applied Mathematics \\
48009 Bilbao, Spain and Ikerbasque, Basque Foundation for Science, 48011 Bilbao, Spain}
\email{lroncal@bcamath.org}

\address[Olli Tapiola]{Department of Mathematics and Statistics, P.O. Box 35 (MaD), FI-40014 University of Jyv\"askyl\"a, Finland}
\email{olli.m.tapiola@gmail.com}

\thanks{
J.C. was supported by Basque Government through Programa de formaci\'on de personal investigador no doctor. J.C. and L.R. were supported by Basque Government through the BERC 2018-2021 program, by the Spanish Ministry of Economy and Competitiveness through the project MTM2017-82160-C2-1-P and by the Spanish Ministry of Science, Innovation and Universities: BCAM Severo Ochoa accreditation SEV-2017-0718.\\
O.T. was partially supported by the V\"ais\"ala Fund through a Finnish Academy of Science and Letters travel grant. }

\allowdisplaybreaks

\begin{document}

\begin{abstract}
  We consider Coifman--Fefferman inequalities for rough homogeneous singular integrals $T_\Omega$ and $C_p$ weights. It was recently shown in \cite{liperezriverariosroncal} that
$$
    \|T_\Omega \|_{L^p(w)} \le C_{p,T,w} \|Mf\|_{L^p(w)}
$$
  for every $0< p < \infty$ and every $w \in A_\infty$. Our first goal is to generalize this result for every $w \in C_q$ where $q > \max\{1,p\}$ without using extrapolation theory.
Although the bounds we prove are new even in a qualitative sense, we also give the quantitative bound with respect to the $C_q$ characteristic.
Our techniques rely on recent advances in sparse domination theory and we actually prove most of our estimates for sparse forms.

  Our second goal is to continue the structural analysis of $C_p$ classes. We consider some weak self-improving properties of $C_p$ weights and weak and dyadic $C_p$ classes. We also revisit and generalize a counterexample by Kahanp\"a\"a and Mejlbro \cite{kahanpaamejlbro} who showed that $C_p \setminus \bigcup_{q > p} C_q \neq \emptyset$. We combine their construction with techniques of Lerner \cite{lerner_remarks} to define an explicit weight class $\widetilde{C}_p$ such that $\bigcup_{q > p} C_q \subsetneq \widetilde{C}_p \subsetneq C_p$ and every $w \in \widetilde{C}_p$ satisfies Muckenhoupt's conjecture \cite{muckenhoupt_norm_inequalities}. In particular, we give a different, self-contained proof for the fact that the $C_{p+\eps}$ condition is not necessary for the Coifman--Fefferman inequality and our ideas allow us to consider also dimensions higher than $1$.
\end{abstract}

\maketitle

\section{Introduction}

It is a long-standing open problem in harmonic analysis to characterize the weights $w$ that satisfy the Coifman--Fefferman inequality
\begin{equation}
  \label{ineq:coifman-fefferman}
  \|Tf\|_{L^p(w)} \le C \|Mf\|_{L^p(w)}
\end{equation}
for a fixed $0 < p < \infty$ and a uniform constant $C$, where $T$ is a singular integral operator and $M$ is the Hardy--Littlewood maximal operator
(see Section \ref{section:definitions} for precise definitions of these and subsequently mentioned objects).
The inequality was first verified for $A_\infty$ weights and maximally truncated Calder\'on--Zygmund operators
by Coifman and Fefferman \cite[Theorem III]{coifmanfefferman} who combined it with Muckenhoupt's theorem \cite{muckenhoupt_weighted_norm}
to conclude that Muckenhoupt's $A_p$ condition implies the uniform weighted $L^p$-boundedness of Calder\'on--Zygmund operators. Over the last decades, Coifman--Fefferman type domination inequalities have had an important role in many advances in harmonic analysis, see e.g. \cite{changwilsonwolff, grafakosmartellsoria, lernerombrosiperez, cruz-uribemartellperez, perezpradolinitorrestrujillo-gonzalez, chillfiorenza}. It was later shown by Muckenhoupt \cite{muckenhoupt_norm_inequalities} that weights satisfying \eqref{ineq:coifman-fefferman}
can actually vanish on a set with infinite measure and thus, the $A_\infty$ condition is too strong to characterize the inequality.
In addition, he showed that if \eqref{ineq:coifman-fefferman} holds for the Hilbert transform, then the
weight has to satisfy the so called \emph{$C_p$ condition}: there exist constants $C, \eps > 0$ such that for every cube $Q$
and every measurable set $E \subseteq Q$ we have
\begin{equation}
  \label{cond:c_p}
  w(E) \le C \left( \frac{|E|}{|Q|} \right)^\eps \int_{\R^n} (M1_Q)^p w.
\end{equation}
He conjectured that this condition is also sufficient for \eqref{ineq:coifman-fefferman}.
Sawyer \cite{sawyer_two_weight,sawyer_norm_inequalities} noted that \eqref{ineq:coifman-fefferman} holds also for weak $A_\infty$ weights, extended Muckenhoupt's
result for the Riesz transforms and gave a partial answer to Muckenhoupt's conjecture: if the weight $w$ satisfies
the $C_{p+\lambda}$ condition for some $\lambda > 0$ (which is stronger than the $C_p$ condition), then \eqref{ineq:coifman-fefferman} holds for
Calder\'on--Zygmund operators. Later, Kahanp\"a\"a and Mejlbro \cite{kahanpaamejlbro} showed that the $C_{p+\lambda}$ condition is not necessary for \eqref{ineq:coifman-fefferman} in dimension $1$, but the full answer to the conjecture is still not known in any dimension. Finally, we note that Martell, P\'erez and Trujillo-Gonz\'alez \cite{martellpereztrujillo-gonzalez} showed via extrapolation methods that there exist singular integral operators that do not satisfy \eqref{ineq:coifman-fefferman} for any $0 < p < \infty$ and any $w \in A_\infty$.

The $C_p$ classes resemble Muckenhoupt's $A_p$ classes in some ways (for example, $C_p$ weights satisfy Reverse H\"older type inequalities) but their overall structure
is much more chaotic. In particular, it was shown by Kahanp\"a\"a and Mejlbro \cite{kahanpaamejlbro} that unlike all $A_p$ weights, some $C_p$ weights do not have
any kind of self-improving property with respect to $p$, i.e. $C_p \setminus \bigcup_{q > p} C_{q} \neq \emptyset$.
Naturally, this and some other unfortunate properties of these weights have made it impossible to use any straightforward $A_p$ type techniques for the problem. However, although Muckenhoupt's conjecture is still open, many authors have managed to study other parts of the $C_p$ theory, see e.g. \cite{yabuta, buckley, delatorreriveros, lerner_remarks, auscherbortzegertsaari}.

In this paper, we have two goals. Our first goal is to prove Sawyer type $C_p$ estimates for \emph{rough homogeneous singular integrals}, i.e. integral operatos $T_\Omega$ defined as
$$
  T_\Omega f(x) \coloneqq \text{p.v.} \int_{\R^n} \frac{\Omega(y')}{|y|^n} f(x-y) \, dy,
$$
where $y' \coloneqq y / |y|$, $\Omega \in L^\infty(\mathbb{S}^{n-1})$  and $\int_{\mathbb S^{n-1}}\Omega \, d\sigma =0.$
These operators have been studied intensively by numerous authors both in unweighted and weighted settings, see e.g.
\cite{duoandikoetxearubiodefrancia, hofmann, christ, christrubiodefrancia, watson, duoandikoetxea_weighted, seeger, hytonenroncaltapiola, grafakoshehonzik}.
Our results complement the recent work of Cejas, the second author, P\'erez and Rivera-R\'ios \cite{cejasliperezrivera-rios} who discussed
Coifman--Fefferman inequalities for these operators in \cite[Remark 5]{cejasliperezrivera-rios}. They can also be seen as a continuation of the work of the second author, Per\'ez, Rivera-R\'ios and the third author \cite{liperezriverariosroncal} who proved these types of estimates for rough homogeneous singular integrals and $A_\infty$ weights, and the work of the first author \cite{canto} who recently introduced the $C_p$ constant $[w]_{C_p}$ (see Subsection \ref{section:weights} for the definition) and studied quantitative Coifman--Fefferman inequalities.

Let us be more precise. We prove the following inequalities:
\begin{theorem}
  \label{thm:main_result_rough}
  Suppose that $T_\Omega$ is a rough homogeneous singular integral with $\Omega \in L^\infty(\Sb^{n-1})$ satisfying $\int_{\Sb^{n-1}} \Omega \, d\sigma = 0$.
  Then the following inequalities hold:
  \begin{enumerate}
    \item[I)] if $1 < p < q < \infty$ and $w \in C_q$, then
              $$
                \| T_\Omega f\|_{L^p(w)} \le C_{n,p,q}  \big([w]_{C_q} + 1 \big)^{3} \log\big( [w]_{C_q} + e \big) \| \Omega \|_{L^\infty} \| Mf \|_{L^p(w)},
          $$

    \item[II)] if $0 < p \le 1 < q < \infty$ and $w \in C_q$, then
            $$
                 \| T_\Omega f\|_{L^p(w)} \le C_{n,p,q}  \big([w]_{C_q} + 1 \big)^{1+\frac 2p} \log^{\frac 1p}\big( [w]_{C_q} + e \big) \|\Omega \|_{L^\infty} \| Mf \|_{L^p(w)}.
             $$
  \end{enumerate}
  The constant $C_{n,p,q}$ satisfies $C_{n,p,q} \to \infty$ as $q \to p$.
\end{theorem}
We want to emphasize that the main novelty of this result is the qualitative estimates that (to the best of our knowledge) were not known earlier. We do not know if our bounds are sharp with respect to $[w]_{C_p}$ but we strongly suspect that they are not. We also note that previous proofs for the case $0 < p < 1$ and $w \in A_\infty$ used extrapolation theory which is not available for $C_p$ weights. Our method and quantitative bounds are new even for weights $w \in A_\infty$.

Our proof relies particularly on a recent sparse domination result of Conde-Alonso, Culiuc, Di Plinio and Ou:
\begin{theorem}[{\cite[part of Theorem A]{condealonsoculiucdiplinioou}}]
  \label{thm:condealonsoetal}
  Suppose that $T_\Omega$ is a rough homogeneous singular integral with $\Omega \in L^\infty(\Sb^{n-1})$ and $\int_{\Sb^{n-1}} \Omega \, d\sigma = 0$. Then, for any $1 < p < \infty$ we have
$$
    |\langle T_\Omega f,g \rangle| \le c_n p' \| \Omega \|_{L^\infty(\Sb^{n-1})} \sup_\Sc \sum_{Q \in \Sc} \langle |f| \rangle_{Q} \langle |g| \rangle_{p,Q},
$$
  where the supremum is taken over all sparse collections $\Sc$ (see Section \ref{section:definitions}).
\end{theorem}
An alternative approach for this result can be found in \cite{lerner_weak-type}. Thus, instead of working directly with rough homogeneous singular integrals, we use Theorem \ref{thm:condealonsoetal} to reduce the question to proving bounds for sparse forms:
\begin{theorem}
  \label{thm:main_sparse_bounds}
  Let $\Lambda = \Lambda_\Sc^{t,\gamma}$ be the sparse form defined as
 $$
    \Lambda(f,g) \coloneqq (t')^\gamma \sum_{Q \in \Sc} \langle |f| \rangle_{Q}^\gamma \langle |g|\rangle_{t,Q}|Q|,
$$
  where $\Sc$ is a sparse collection of cubes, $t > 1$ and $0 < \gamma \le 1$.
  \begin{enumerate}
    \item[I)] Suppose that $1 < p < q < \infty$ and $w \in C_q$. Then there exists
              $1 < s < 2$ such that
           $$
                \Lambda_\Sc^{s,1}(f,gw) \le C_{n,p,q}   \big([w]_{C_q} + 1 \big)^{3} \log\big( [w]_{C_q} + e \big)  \|Mf\|_{L^p(w)} \|g\|_{L^{p'}(w)}.
      $$

    \item[II)] Suppose that $0 < p \le 1<q<\infty$ and $w \in C_q$. Then  there exists $1 < s < \min\{2,\frac{1}{1-p}\}$ such that
             $$
                 \Lambda_\Sc^{s,p}(f,w) \le C_{n,p,q} \big([w]_{C_q} + 1 \big)^{ p+ 2} \log \big( [w]_{C_q} + e \big) \|Mf\|_{L^p(w)}^p .
              $$
  \end{enumerate}
  The constant $C_{n,p,q}$ satisfies $C_{n,p,q} \to \infty$ as $q \to p$.
\end{theorem}
Part I) of Theorem \ref{thm:main_result_rough} follows from Theorem \ref{thm:condealonsoetal} and part I) of Theorem \ref{thm:main_sparse_bounds} in a very straightforward way but for part II)  we need some additional considerations. In particular, we need to modify some results proven by Lerner \cite{lerner_weak-type} and prove a variation of the sparse domination result for the case $0 < p < 1$ (see Theorem \ref{thm:sparse_domination}).

We note that in \cite{condealonsoculiucdiplinioou}, the authors proved similar sparse domination results also for other classes of operators, namely rough homogeneous singular integrals $T_\Omega$ with more general kernel functions $\Omega$ and Bochner--Riesz means. Their results combined with Theorem \ref{thm:main_sparse_bounds} give $C_q$-Coifman--Fefferman estimates also for these operators for $1 \le p < \infty$ but for simplicity, we only consider the operators $T_\Omega$ with $\Omega \in L^\infty(\Sb^{n-1})$ satisfying $\int_{\Sb^{n-1}} \Omega \, d\sigma = 0$.

Our second goal is to continue the structural analysis of $C_p$ classes started particularly by Buckley \cite[Section 7]{buckley}. We consider:
\begin{enumerate}
  \item[i)] weak self-improving properties of $C_p$ weights,
  \item[ii)] weak and dyadic $C_p$ classes,
  \item[iii)] examples of $C_p$ weights,
  \item[iv)] $C_\psi$ classes of Lerner \cite{lerner_remarks},
  \item[v)] generalizations of the Kahanp\"a\"a--Mejlbro counterexample \cite{kahanpaamejlbro}.
\end{enumerate}
As a corollary of our considerations we are able to give a new proof for the fact that the $C_{p+\lambda}$ condition is not necessary for \eqref{ineq:coifman-fefferman} in any dimension; see Corollary \ref{cor:not_nec} and Theorem \ref{thm:anyd}.

Our motivation for this analysis comes particularly from the fact that most known $C_p$ techniques are heavy or very restricted and this is mainly because not many characterizations and non-trivial properties of $C_p$ weights are known. Naturally, the $C_p$ theory cannot be as comprehensive and rich as the $A_p$ theory because the $C_p$ classes are much bigger than the $A_p$ classes. However, existing results already show that at least some parts of the $A_p$ theory have counterparts in the $C_p$ world.

This paper is structured as follows. First, we introduce the notation and some definitions that will be used throughout the paper. In Section \ref{section:Marcinkiewicz} we present some $C_p$ techniques that will be useful for us later. Sections \ref{section:proof_part_i} and \ref{section:proof_part_ii}  are devoted to the proofs of Theorems \ref{thm:main_result_rough} and \ref{thm:main_sparse_bounds}. In Section \ref{section:sparse} we state and prove a sparse domination result for rough singular integrals that is useful in the range $0<p\leq 1$. In Sections \ref{section:RH} and \ref{section:weak_cp} we consider structural properties of $C_p$ classes and new classes weak $C_p$ and dyadic $C_p$, which are actually equal to $C_p$. Finally, in Section \ref{section:kahanpaamejlbro} we revisit results of Kahanp\"a\"a and Mejlbro \cite{kahanpaamejlbro} and Lerner \cite{lerner_remarks} to show that there exists a weight class $\widetilde{C}_p$ such that $\bigcup_{q > p} C_q \subsetneq \widetilde{C}_p \subsetneq C_p$ and the $\widetilde{C}_p$ condition implies \eqref{ineq:coifman-fefferman}. Unlike some earlier considerations related to this topic, our ideas work in any dimension.

\section{Notation and definitions}
\label{section:definitions}

We use the following notation and terminology in the paper.
\begin{enumerate}
  \item[$\bu$] The letters $c$ and $C$ denote constants that depend only on the dimension and
               other similar parameters. We call them \emph{structural constants}.
               The values of $c$ and $C$ may change from one occurence to another. In most cases, we do
               not track how our bounds depend on these constants and usually just write
               $\gamma_1 \lesssim \gamma_2$ if $\gamma_1 \le c\gamma_2$ for a structural
               constant $c$ and $\gamma_1 \approx \gamma_2$ if $\gamma_1 \lesssim \gamma_2 \lesssim \gamma_1$.
               If the constant $c_\kappa$ depends only on structural constants and some other parameter
               $\kappa$ and $\gamma_1 \lesssim c_\kappa \gamma_2$, we write $\gamma_1 \lesssim_\kappa \gamma_2$.

  \item[$\bu$] The Lebesgue measure of a measurable set $E \subset \R^n$ is denoted by $|E|$.

  \item[$\bu$] The characteristic function of a set $E$ is denoted by $1_E$.

  \item[$\bu$] A \emph{weight} is a non-negative locally Lebesgue integrable function
               that is non-zero in a set of positive measure.

  \item[$\bu$] Suppose that $f$ is a locally integrable function, $w$ is a weight, $E \subset \R^n$
               is a measurable set with $|E| > 0$ and $0 < p < \infty$. We denote
               \begin{align*}
                 w(E) &\coloneqq \int_E w \, dx,\quad
                 \langle f \rangle_E := \avgint_E f \, dx \coloneqq \frac{1}{|E|} \int_E f \, dx, \quad  \langle |f| \rangle_{p,E} := \langle |f|^p \rangle_E^{1/p}\\
                 \langle f \rangle_{E}^w &\coloneqq \frac{1}{w(E)} \int_E fw \, dx.
                 %
               \end{align*}

  \item[$\bu$] For any cube $Q$   we denote the
               side length of $Q$  by $\ell(Q)$.

  \item[$\bu$] The \emph{collection of dyadic cubes} $\Ds$ is defined as
               $$
                 \Ds \coloneqq \{2^{-k} ( [0,1)^n + z ) \colon k \in \Z, z \in \Z^n \}.
             $$

  \item[$\bu$] Suppose that $0 < \gamma < 1$. A \emph{$\gamma$-sparse collection} $\Sc$ is a
               collection of cubes such that for every $Q \in \Sc$ there exists a $E_Q \subset Q$ such that
               \begin{enumerate}
                 \item[(1)] $|E_Q| \ge \gamma |Q|$,
                  \item[(2)] if $Q, Q' \in \Sc$ and $Q \neq Q'$, then $E_Q \cap E_{Q'} = \emptyset$.
               \end{enumerate}
               In most cases, we assume that $\gamma = \tfrac{1}{2}$ and do not specifically mention this every time. We note that we do not require our sparse collections to be subcollections of $\Ds$ but any $\gamma$-sparse collection can be embedded inside a bounded number of dyadic $\tfrac{\gamma}{6^n}$-sparse collections. This follows from e.g. \cite[Lemma 2.5]{hytonenlaceyperez}.

  \item[$\bu$] We say that a function $f \colon \R^n \to \R$ is \emph{lower semicontinuous} if the set $\{x \in \R^n \colon f(x) > \lambda\}$ is open for every $\lambda \in \R$.
\end{enumerate}

\subsection{$A_p$, $A_\infty$, $C_p$ and Reverse H\"older weights}
\label{section:weights}

Suppose that $w$ is a weight. We denote $w \in A_p$ for $1 < p < \infty$ if
$$
 [w]_{A_p} \coloneqq \sup_Q \Big( \avgint_Q w \Big) \Big( \avgint_Q w^{-\frac{1}{p-1}} \Big)^{p-1} < \infty,
$$
where the supremum is taken over all cubes. We say that $w$ satisfies a \emph{$q$-Reverse H\"older inequality}
for $1 < q < \infty$ and denote $w \in RH_q$ if there exists a constant $[w]_{RH_q} < \infty$ such that
$$
  \Big( \avgint_Q w^q \Big)^{1/q} \le [w]_{RH_q} \avgint_Q w
$$
for every cube $Q$. We denote $w \in A_\infty$ if the \emph{Fujii--Wilson constant} of $w$ is finite, i.e.
$$
  [w]_{A_\infty} \coloneqq \sup_Q \frac{1}{w(Q)} \int_Q M(1_Q w) < \infty.
$$
By \cite{coifmanfefferman, fujii, wilson}, we know that the following conditions are equivalent:
\begin{enumerate}
  \item[i)] $w \in A_\infty$,
  \item[ii)] $w \in \bigcup_{p > 1} A_p$,
  \item[iii)] $w \in \bigcup_{q > 1} RH_q$,
  \item[iv)] there exist constants $C, \eps > 0$ such that for every cube $Q$ and every measurable set $E \subseteq Q$ we have
             \begin{equation}
               \label{defin:a_infty}
               w(E) \le C \Big( \frac{|E|}{|Q|} \Big)^\eps w(Q).
             \end{equation}
\end{enumerate}
Recall from \eqref{cond:c_p} that $w \in C_p$ for $1 < p < \infty$ if there exist constants $C, \eps > 0$ such that for every cube $Q$ and every measurable set $E \subseteq Q$ we have
$$
  w(E) \le C \Big( \frac{|E|}{|Q|} \Big)^\eps \int_{\R^n} (M1_Q)^p w.
$$
Recently, the first author \cite{canto} introduced a Fujii--Wilson type $C_p$ characteristic. Let $w \in C_p$.
\begin{enumerate}
 \item[$\bu$] If $\int_{\R^n} (M1_Q)^p w = \infty$ for some (and thus, all) cubes $Q$, we set $[w]_{C_p} = 0$.
 \item[$\bu$] If $0 < \int_{\R^n} M(1_Q)^p w < \infty$ for some (and thus, all) cube $Q$, we set
             $$
                [w]_{C_p} = \sup_Q \frac{1}{\int_{\R^n} (M1_Q)^p w} \int_Q M(1_Q w).
             $$
\end{enumerate}

\begin{theorem}[{\cite[Section 2]{canto}}]
  \label{thm:c_p_constant}
  Let $1 < p < \infty$. We have $w \in C_p$ if and only if $[w]_{C_p} < \infty$.
\end{theorem}
We note that although $[w]_{A_\infty} \ge 1$ for any $w \in A_\infty$, the characteristic $[w]_{C_p}$ can be arbitrarily small \cite[Example 2.11]{canto}. When comparing the definition \eqref{cond:c_p} and the characterization in Theorem \ref{thm:c_p_constant}, it was noted in \cite[Remark 2.16]{canto} that if $0 \le [w]_{C_p} < \infty$, then \eqref{cond:c_p} holds for $C = 2$ and
\begin{equation}
  \label{quantity:eps_cp}
  \eps = \frac{1 - 2^{-n(p-1)}}{2^{2np + 3n}(20)^n} \min\{1,[w]_{C_p}^{-1}\}.
\end{equation}
In particular, the constant $C$ in the definition \eqref{cond:c_p} is fairly irrelevant.

\begin{remark}\label{Remark-finiteness}
If $[w]_{C_q}=0$, then $\int_{\R^n} (Mf)^pw= \infty$ for all $p\leq q$ and all nonzero function $f$. Since most of the estimates we are going to prove have $\| Mf\|_{L^p(w)}$ on the right hand side, we may always assume without loss of generality that $[w]_{C_q}>0$.
\end{remark}

The first author also proved a quantitative Reverse H\"older type estimate for $C_p$ weights (see also \cite[Section 8.1]{auscherbortzegertsaari} and \cite[Lemma 7.7]{buckley} for qualitative Reverse H\"older type estimates for $C_p$ weights):

\begin{theorem}[{\cite[Section 2]{canto}}]
  \label{thm:rhi_cp}
  Suppose that $1 < p < \infty$ and $w \in C_p$. Then there exists a constant $B = B(n,p)$ such that
  \begin{equation}
    \label{ineq:reverse_holder_cp}
    \Big( \avgint_Q w^{1+\delta} \Big)^{\frac{1}{1+\delta}} \le \frac{C}{|Q|} \int_{\R^n} (M1_Q)^p w
  \end{equation}
  for a structural constant $C$, every cube $Q$ and every $0 < \delta \le \tfrac{1}{B \max\{[w]_{C_p},1\}}$. Conversely, if there exist constants $C > 0$ and $\delta > 0$ such that \eqref{ineq:reverse_holder_cp} holds for every cube $Q$, then $w \in C_p$.
\end{theorem}

\subsection{Maximal functions and singular integrals}
\label{section:maximal_and_singular}

Suppose that $f$ is a locally integrable function. The \emph{Hardy--Littlewood maximal operator} $M$ is defined as
$$
  Mf(x) \coloneqq \sup_{Q \ni x} \avgint_Q |f(y)| \, dy,
$$
where the supremum is taken over all cubes $Q$ that contain $x$. For every $s \ge 1$, we define the \emph{$s$-maximal operator} $M_s$ as
$$
  M_sf(x) \coloneqq \big( M(|f|^s \big)^{1/s}.
$$
Since the Hardy--Littlewood maximal operator is of weak type $(1,1)$, i.e.
$$
  \|Mf\|_{L^{1,\infty}} \le C_n \|f\|_{L^1}
$$
for every $f \in L^1(\R^n)$, it is straightforward to check that $M_s$ is of weak type $(s,s)$, i.e.
$$
  \|M_sf\|_{L^{s,\infty}} \le C_n^{1/s} \|f\|_{L^s}
$$
for every $f \in L^s(\R^n)$. In our constructions and proofs, we use repeatedly the fact that the Hardy--Littlewood maximal function $Mf$ is lower semicontinuous (see e.g. \cite[proof of Theorem 2.1.6]{grafakos}).

Let $T$ be a bounded linear operator on $L^2(\R^n)$ that has the representation
$$
  Tf(x)=\int_{\R^d}K(x,y)f(y)\,dy
$$
for every $x \notin \text{supp} \, f$. We say that $T$ is a \emph{Calder\'on--Zygmund operator}
if the kernel function $K$ satisfies the size estimate
$$
  |K(x,y)| \le \frac{C_K}{|x-y|^n}
$$
for all $x, y \in \R^n$, $x \neq y$, and the smoothness estimate
$$
  |K(x,y)-K(x',y)|+|K(y,x)-K(y,x')|\le C_K \frac{|x-x'|^\lambda}{|x-y|^{n+\lambda}}
$$
for some $\lambda>0$ and all $x,x',y \in \R^n$ such that $|x-y| > 2|x-x'| > 0$.

\section{Marcinkiewicz integral estimates} \label{section:Marcinkiewicz}

We start by recalling and refining some estimates related to Marcinkiewicz integral operators. These operators are one of the key tools in $C_p$ analysis due to their good boundedness properties with respect to certain weights. For our needs, the most convenient way to define the operators is using a Whitney decomposition of level sets as in \cite[Section 3]{sawyer_norm_inequalities}.

\begin{lemma}[{\cite{sawyer_norm_inequalities}}]
  Suppose that $R \ge 1$ and $\Omega \subset \R^n$ is an open set. Then there exists a constant $C = C(R,n)$ independent of $\Omega$ and a collection of disjoint cubes $\Qc \coloneqq \{Q_j\}_j$ such that
  \begin{enumerate}
    \item[i)] $\Omega = \bigcup_j Q_j$,

    \item[ii)] for every $j$ we have
              $$
                 5R \le \frac{\dist(Q_j, \Omega^c)}{\diam \, Q_j} \le 15R,
             $$

    \item[iii)] $\sum_j 1_{R Q_j}(x) \le C 1_\Omega(x)$ for every $x \in \R^n$.
  \end{enumerate}
\end{lemma}

Let $h$ be a non-negative lower semicontinuous function, $0 < p,q < \infty$, and $k \in \Z$. Then the level set $\Omega_k =\{x: h(x)>2^k\}$ is open and we can use the previous lemma to get the decomposition $\Omega_k = \bigcup_j Q_j^k$. We denote $\Qc_k \coloneqq \{Q_j^k\}_j$ for each $k \in \Z$ and define the Marcinkiewicz integral operator $M_{p,q}$ by setting
$$
  M_{p,q} h(x)  \coloneqq  \Big( \sum_{k \in \Z} 2^{kp} \sum_{Q \in \Qc_k}  M1_Q(x)^q\Big)^{\frac{1}{p}}.
$$
Note that the dependency on $h$ on the right hand side is coded in the families $\mathcal Q_k$. For each $k \in \Z$, we define the partial Marcinkiewicz operator $M_{k,p,q}$ the same way as in \cite{cejasliperezrivera-rios}:
$$
  M_{k,p,q} h(x) \coloneqq \Big( 2^{kp} \sum_{Q \in \mathcal Q_k}  M1_Q(x)^q \Big)^{\frac{1}{p}} .
$$
Thus, we have
$$
  \sum_{k \in \Z} M_{k,p,q}h(x)^p = M_{p,q} h(x)^p.
$$
These operators arise naturally when estimating $L^p(w)$ norms with $w \in C_q$. Indeed, by the layer cake representation \cite[Proposition 1.1.4]{grafakos} , we have
$$
  \| h \|_{L^p(w)}^p = p \int_0^\infty t^{p-1} w(\{ h>t \}) dt \approx p \sum_{k\in \Z} 2^{kp} \sum_{Q \in \mathcal Q_k} w(Q).
$$
The role that $w(Q)$ plays in the $A_\infty$ theory is often played by $\int_{\R^n} M(1_Q)^qw$ in the $C_q$ context. Therefore, the natural $C_q$ counterpart of the above expression is
$$
  \sum_{k\in  \Z} 2^{kp} \sum_{Q\in \mathcal Q_k} \int_{\R^n}(M1_Q)^q w = \int_{\R^n} (M_{p,q}h)^p w.
$$

The proof of the following lemma can be found in \cite[Lemma 5.8]{canto}. Although the range of exponents is not explicitly stated there, it holds for all exponents described below.
\begin{lemma}
  \label{lemma-Marcinkiewicz}
  Let $f$ be a compactly supported function and $w \in C_q$ with $1 < q <\infty$. Suppose that $0 < p < q$. Then
  $$
    \int_{\R^n} (M_{p,q} (Mf))^p w \leq c_n c_{p,q} \max(1,[w]_{C_q} \log [w]_{C_q}) \int_{\R^n} (Mf)^p w.
$$
\end{lemma}

\begin{lemma}
  \label{summation-sparse-lemma}
  Let $Q$ be a cube and $\Sc$ a sparse family of cubes that are contained in $Q$. Suppose that $w \in C_q$ with $1 < q < \infty$. Then
 $$
    \int_{\R^n} \sum_{R \in \Sc} \big( M1_R \big)^q w \lesssim \big( [w]_{C_q}+1\big)  \int_{\R^n} (M1_Q)^q w.
 $$
\end{lemma}

\begin{proof}
We start by noticing that if $x \notin 2Q$, then we have
\begin{align*}
  \sum_{R\in \Sc} M1_R(x)^q
    &\lesssim \sum_{R \in \Sc}\Big(  \frac{|R|}{\dist (x,Q)^n} \Big)^q
    \lesssim \sum_{R\in \Sc} \Big( \frac{|E_R|}{\dist(x,Q)^n}\Big)^q \\
    &= \frac{\sum_{R\in \Sc} |E_R|^q}{\dist(x,Q)^{nq}}  \le \Big( \frac{|Q|}{\dist(x,Q)^n}\Big)^q  \lesssim M1_Q(x)^q,
\end{align*}
where $E_R$ is the exceptional set given by sparsity and we used the assumption $q > 1$ in the estimate $\sum_{R \in \Sc} |E_R|^q \le |Q|^q$. Thus, it is enough to bound $\int_{2Q} \sum_{R \in \Sc} \big( M1_R \big)^q w$.

Since $E_R \subset R$ and $|E_R| \ge \tfrac{1}{2}|R|$ for every $R \in \Sc$, we have the pointwise bound
$$
  \sum_{R \in \Sc} (M1_R)^q \lesssim \sum_{R \in \Sc} (M1_{E_R})^q
$$
by Lemma \ref{lemma:hl_bound}. Also, since the sets $E_R$ are pairwise disjoint, we have $\sum_R (1_{E_R})^q \le 1 \in L^\infty$. Thus,
by \cite[Theorem 1 (3)]{feffermanstein_some} there exists $c > 0$ such that
for every $\lambda>0$ we have
\begin{equation}
  \label{estimate:level_set}
  |F_\lambda| \coloneqq |\{x\in 2Q: \sum M1_R(x)^q > \lambda \}| \le c e^{-c \lambda} |Q|.
\end{equation}
Applying the $C_q$ condition to $F_\lambda\subseteq 2Q$ now gives us
\begin{equation}
  \label{estimate:F_lambda}
  w(F_\lambda) \le c \Big( \frac{|F_\lambda|}{|2Q|} \Big)^\eps \int_{\R^n}( M1_{2Q})^q w
               \overset{\eqref{quantity:eps_cp}, \eqref{estimate:level_set}}{\lesssim} e^{-c \frac{\lambda}{[w]_{C_q}+1}} \int_{\R^n}( M1_Q)^q w.
\end{equation}
Thus, for any fixed $\lambda > 0$ we have
$$
  \int_{2Q} \sum_{R \in \Sc} (M1_R)^q w
  = \int_0^\infty w(F_t) \, dt
  = \int_0^\lambda w(F_t) \, dt + \int_\lambda^\infty  w(F_t) \, dt
  \eqqcolon I_1 + I_2.
$$
For $I_1$, we can use Lemma \ref{lemma:hl_bound} to see that
$$
  I_1 \le \lambda w(2Q)
      = \lambda \int_{\R^n} 1_{2Q} w
      \le \lambda \int_{\R^n} M(1_{2Q})^q w
      \lesssim \lambda \int_{\R^n} M(1_Q)^q w.
$$
For $I_2$, we can use \eqref{estimate:F_lambda} to get
$$
  I_2 \lesssim \int_\lambda^\infty e^{- c \frac{t}{[w]_{C_q}+1}} dt  \int_{\R^n} (M1_Q)^q w
      \le c([w]_{C_q}+1) e^{-c \frac{\lambda}{[w]_{C_q}+1}}  \int_{\R^n} (M1_Q)^q w.
$$
Thus, we have
$$
  I_1 + I_2 \lesssim \Big(\lambda + \big( [w]_{C_q}+1 \big) e^{-c \frac{\lambda}{[w]_{C_q}+1}} \Big)  \int_{\R^n} (M1_Q)^q w
$$
and choosing $\lambda = [w]_{C_q}+1$ completes the proof.
\end{proof}

\begin{lemma}
  \label{lemma-sparse}
  Let $h$ be a non-negative lower semicontinuous function, $w \in C_q$, $1 < q < \infty$ and $0 < p < \infty$. Suppose that $k \in \Z$ and let $\Sc = \{R_j\}$ be a sparse collection of cubes contained in $\Omega_k = \{ x \colon h(x) > 2^k\}$. Then
 $$
    2^{kp} \sum_{R_j \in \Sc} \int_{\R^n} (M1_{R_j})^q w \lesssim \big( [w]_{C_q} + 1 \big) \int_{\R^n} (M_{k,p,q} h)^p w
$$
\end{lemma}

\begin{proof}
Fix $k\in \Z$ and let $\Qc_k = \{Q_l\}_l$ be the Whitney decomposition of $\Omega_k$. For each $Q_l\in \Qc_k$, let $\Sc _{k,l}$ be the family of cubes $R_j$ whose center is contained in $Q_l$. Then, by the properties of the Whitney cubes and the fact that $R_j \subset \Omega_k$, we have $R_j \subset c_n Q_l$ for every $R_j \in \Sc_{k,l}$. Moreover, each $R_j\in \Sc$ is contained in exactly one of the $\Sc_{k,l}$.

The desired estimate follows now from applying Lemma \ref{summation-sparse-lemma} to each of the collections $\Sc_{k,l}$:
\begin{align*}
  2^{kp}\sum_{R_j \in \Sc} \int_{\R^n} (M1_{R_j})^q w
  &= 2^{kp} \sum_{Q_l \in \Qc_k} \sum_{R_j \in \Sc_{k,l}} \int_{\R^n} (M1_{R_j})^q w \\
  &\lesssim \big( [w]_{C_q} + 1 \big) 2^{kp} \sum_{Q_l \in \Qc_k} \int_{\R^n} (M1_{Q_l}) ^q w \\
  & = \big( [w]_{C_q}+1\big)  \int_{\R^n} (M_{k,p,q}h)^p w. \qedhere
\end{align*}
\end{proof}

\begin{corollary}
  \label{Corollary-norm}
  Suppose that $\Sc$ is a sparse collection of cubes, $f$ is a locally integrable function, $w\in C_q$ for $1 < q < \infty$ and $0 < p < q.$ Then
 $$
    \sum_{Q\in \Sc} \langle f \rangle_Q^p \int_{\R^n} (M1_Q)^q w \le c_{n,p,q} ([w]_{C_q}+1)^2\log ([w]_{C_q}+e) \| Mf\|_{L^p(w)}^p.
 $$
\end{corollary}

\begin{proof}
We start by making a level decomposition of the sparse family: for every $k \in \Z,$ we set
$$
\Sc_k \coloneqq \{ Q\in \Sc : 2^k < \langle |f| \rangle_Q \le 2^{k+1}\}.
$$
Clearly we have $\Sc = \bigcup_{k \in \Z} \Sc_k$. Now, for each $Q \in \Sc_k$, we have trivially $Q\subset \{ Mf>2^k\}$. Thus, Lemmas \ref{lemma-sparse} and \ref{lemma-Marcinkiewicz} give us
\begin{align*}
\sum_{Q \in \Sc} \langle |f| \rangle_Q^p \int_{\R^n} (M1_Q)^q w
&\le 2^p \sum_{k \in \Z} 2^{kp} \sum_{Q \in \Sc_k} \int_{\R^n} (M1_Q)^q w \\
&\lesssim 2^p \big( [w]_{C_q} + 1\big)  \sum_{k \in \Z} \int_{\R^n} (M_{k,p,q} Mf)^p w. \\
&=  2^p \big( [w]_{C_q} + 1 \big) \int_{\R^n} (M_{p,q} Mf)^p w \\
&\le c_{n,p,q} \big( [w]_{C_q} + 1 \big)^2 \log \big( [w]_{C_q} + e\big) \| Mf \|_{L^p(w)}^p.  \qedhere
\end{align*}
\end{proof}

\section{Proof of part I) of Theorems \ref{thm:main_result_rough} and \ref{thm:main_sparse_bounds}}
\label{section:proof_part_i}

As we stated before, part I) of Theorem \ref{thm:main_result_rough} follows easily from a combination of part I) Theorem \ref{thm:main_sparse_bounds} and Theorem \ref{thm:condealonsoetal}. Indeed, let $s$ be the one given by  Theorem \ref{thm:main_sparse_bounds}. We apply Theorem \ref{thm:condealonsoetal} with parameter $s$ and we get
\begin{align*}
\| T_\Omega f\|_{L^p(w)}  = \sup_{\| g\|_{L^{p'}(w)}=1}  | \langle T_\Omega f, gw\rangle |   
& \leq c_n \| \Omega\|_\infty \: s' \sup_{\| g\|_{L^{p'}(w)}=1} \sup_\Sc \sum_{Q\in \Sc } \langle f\rangle _Q \langle  gw\rangle_{s,Q} |Q| \\
& \leq c_{n,p,q}  \| \Omega\|_\infty \big([w]_{C_q} + 1 \big)^{3} \log\big( [w]_{C_q} + e \big) \|Mf\|_{L^p(w)},
\end{align*}
where we used part I) of Theorem \ref{thm:main_sparse_bounds} in the last inequality.

We now give the proof of part I) of Theorem \ref{thm:main_sparse_bounds}. Let us start by recalling the dyadic Carleson embedding theorem that we need a couple of times in our proofs.

\begin{theorem}[{\cite[Theorem 4.5]{hytonenperez_sharp}}]
  \label{thm:carleson_embedding}
  Let $\mathcal{D}$ be a collection of dyadic cubes, $w$ a weight and $a_Q$ a non-negative number for every $Q \in \mathcal{D}$. Suppose that there exists $A \ge 0$ such that for every $R \in \mathcal{D}$ we have
 $$
    \sum_{ Q \in \mathcal{D}, Q \subset R} a_Q \leq A w(R).
 $$
  Then, for all $1 < \alpha < \infty$ and $h \in L^\alpha(w)$, we have
$$
    \Big( \sum_{R \in \mathcal{D}} a_R ( \langle h \rangle_R^w)^\alpha \Big)^\frac{1}{\alpha} \le A^\frac{1}{\alpha} \cdot \alpha' \cdot \|h\|_{L^\alpha(w)}.
$$
\end{theorem}

Let us then prove part I) of Theorem \ref{thm:main_sparse_bounds}. Suppose that $1 < p < q < \infty$, and $w \in C_q$. We want to show that there exists $1 < s < 2$ such that
$$
  s'\sum_{Q\in \mathcal S} \langle |f|\rangle_{Q}  \langle |gw|\rangle_{s,Q}|Q| \le c_{w,n,p,q} \|Mf\|_{L^p(w)} \|g\|_{L^{p'}(w)}.
$$
By rescaling we may assume that $\|Mf\|_{L^p(w)} = \|g\|_{L^{p'}(w)} = 1$. To simplify the notation, we also assume $f,g \geq 0$. By the remark we made in Section \ref{section:definitions} when we defined sparse collections, we may assume that $\Sc \subset \Ds$.

Let $\delta$ be the Reverse H\"older constant from Theorem \ref{thm:rhi_cp} and set $s = 1 +\frac{\delta}{8p}$ and $r = 1 + \frac{1}{4p}$. It is easy to check that
\begin{equation}
  \label{estimate:s_and_r}
  sr<1+\frac 1{2p}<p' \quad \mbox{and} \quad  \Big(s - \frac{1}{r}\Big)r' = s + \frac{s-1}{r-1}<1+\delta.
\end{equation}
In particular, $(s-\frac 1r)r'$ is an admissible exponent for the Reverse H\"older inequality in Theorem~\ref{thm:rhi_cp}. Therefore, by H\"older's inequality and Theorem \ref{thm:rhi_cp} we have
\begin{align*}
  \sum_{Q \in \Sc} \langle f \rangle_{Q}  \langle g w\rangle_{s,Q}|Q|
  &\le \sum_{Q\in \Sc} \langle f \rangle_{Q}  \langle g ^{sr}w\rangle_{Q}^{\frac 1{sr}} \langle w^{(s-\frac 1r)r'}\rangle_Q^{\frac 1{sr'}}|Q|\\
  &\lesssim  \sum_{Q\in \Sc} \langle f \rangle_{Q}  \langle g ^{sr}w\rangle_{Q}^{\frac 1{sr}} \Big(\frac 1{|Q|} \int_{\R^n} (M1_Q)^q w \Big)^{1-\frac{1}{sr}} |Q|\\
  &=\sum_{Q \in \Sc} \langle f \rangle_{Q}   \Big( \int_{\R^n} (M1_Q)^q w \Big)^{1-\frac 1{sr}} \big(\langle g ^{sr}\rangle_{Q}^w \big)^{\frac {1}{sr}} w(Q)^{\frac 1{sr}}.
\end{align*}
Let us then split the sparse family into two parts. We set
$$
  \Sc_1 \coloneqq \Big\{Q \in \Sc \colon \big(\langle g ^{sr}\rangle_{Q}^w\big)^{\frac 1{sr}}  w(Q)^{\frac{1}{sr}} \le \langle f \rangle_Q^{\frac{p}{p'}} \Big(\int_{\R^n} (M1_Q)^q w\Big)^{\frac{1}{sr}} \Big\}
$$
and $\Sc_2 = \Sc \setminus \Sc_1$. For the collection $\Sc_1$, we use Corollary \ref{Corollary-norm} to see that
\begin{align*}
  \sum_{Q \in \Sc_1} \langle f \rangle_{Q} & \Big( \int_{\R^n} (M1_Q)^q w \Big)^{1-\frac 1{sr}} \big(\langle g ^{sr}\rangle_{Q}^w\big)^{\frac 1{sr}} w(Q)^{\frac 1{sr}} \\
  &\le \sum_{Q \in \Sc_1} \langle f \rangle_{Q} \Big( \int_{\R^n} (M1_Q)^q w \Big)^{1-\frac{1}{sr}} \langle f \rangle_Q^{\frac{p}{p'}} \Big( \int_{\R^n} (M1_Q)^q w \Big)^{\frac{1}{sr}} \\
  &= \sum_{Q\in \Sc_1} \langle f \rangle_{Q}^p\int_{\R^n} (M1_Q)^q w \\
  &\le c_{n,p,q} ([w]_{C_q}+1)^2\log ([w]_{C_q}+e) \| Mf\|_{L^p(w)}^p \\
  &= c_{n,p,q} ([w]_{C_q}+1)^2\log ([w]_{C_q}+e).
\end{align*}
The collection $\Sc_2$ is trickier. Recall that by Remark \ref{Remark-finiteness}, for any cube $Q$, $\int_{\R^n} (M1_Q)^q w < \infty$. Thus, we have
\begin{align*}
  \sum_{Q \in \Sc_2} \langle f \rangle_{Q} & \big(\langle g^{sr} \rangle_{Q}^w\big)^{\frac {1}{sr}}  \Big( \int_{\R^n} (M1_Q)^q w \Big)^{1-\frac 1{sr}}w(Q)^{\frac 1{sr}} \\
  &\le \sum_{Q \in \Sc_2} \big(\langle g^{sr} \rangle_{Q}^w\big)^{\frac{p'}{psr}}w(Q)^{\frac{p'}{psr}} \big(\langle g^{sr} \rangle_{Q}^w\big)^{\frac {1}{sr}}  \Big( \int_{\R^n} (M1_Q)^q w \Big)^{1 - \frac 1{sr} - \frac{p'}{psr}} w(Q)^{\frac 1{sr}} \\
  &\le \sum_{Q \in \Sc_2} \big(\langle g^{sr} \rangle_{Q}^w\big)^{\frac{p'}{sr}} w(Q) \Big(\frac{w(Q)}{\int_{\R^n} (M1_Q)^q w}\Big)^{\frac{p'}{psr}+\frac{1}{sr}-1} \\
  &= \sum_{Q \in \Sc_2} \big(\langle g^{sr} \rangle_{Q}^w\big)^{\frac{p'}{sr}} w(Q) \Big(\frac{w(Q)}{\int_{\R^n} (M1_Q)^q w}\Big)^{\frac{p'}{sr}-1}.
\end{align*}
We set $\alpha = \frac{p'}{sr}$ and
$$
  a_Q := w(Q) \Big(\frac{w(Q)}{\int_{\mathbb R^n}M(\chi_Q)^q w}\Big)^{\frac{p'}{sr}-1}
$$
for every cube $Q \in \Sc_2$. By \eqref{estimate:s_and_r}, we know that $\alpha > 1$. Suppose that there exists some $A>0$ such that for any $R \in \mathcal S_2$ we have
\begin{equation}
  \label{claim:carleson}
  \sum_{Q \in \Sc_2, Q \subset R} a_Q \le A w(R).
\end{equation}
Then, by the Carleson embedding (Theorem \ref{thm:carleson_embedding}), we know that
\begin{align*}
  \sum_{Q \in \Sc_2} \big(\langle g^{sr} \rangle_{Q}^w\big)^{\frac{p'}{sr}} w(Q) \Big(\frac{w(Q)}{\int_{\R^n} (M1_Q)^q w}\Big)^{\frac{p'}{sr}-1}
  &= \sum_{Q \in \Sc_2} a_Q \big(\langle g^{sr} \rangle_{Q}^w\big)^{\alpha} \\
  &\le \big( A^{\frac{1}{\alpha}} \alpha' \| g^{sr}\|_{L^\alpha(w)} \big)^\alpha
  = A (\alpha ')^{\alpha} \|g\|_{L^{p'}(w)}^{p'}
  \le c_p \: A .
\end{align*}
In the last inequality we have used that, by the choices of $r$ and $s$, we have $1<rs<1+\frac1{2p}$ and therefore $p'-rs>p'-1- \frac 1{4p}=\frac{3p+1}{4p(p-1)},$ which gives
\[
\Big( \Big( \frac{p'}{rs} \Big)'\Big) ^\frac{p'}{rs} = \Big( \frac{p'}{p'-rs}\Big) ^{p'} \leq \Big(\frac{4p^2}{3p+1}\Big)^{p'}=c_p .
\]
Thus, it is enough for us to show that there exists a constant $A >0$ such that \eqref{claim:carleson} holds. For this, fix $R \in \Sc_2$. We further split $\Sc_2$ into subcollections $\Sc_{2,j}$, $j \ge 1$, defined as
$$
  \Sc_{2,j} \coloneqq \Big\{ Q \in \Sc_2 \colon 2^{j-1} w(Q) \le \int_{\R^n} (M1_Q)^q w < 2^{j} w(Q) \Big\}.
$$
Let $\Sc_{2,j}^* = \Sc_{2,j}^*(R)$ be the collection of maximal subcubes in $\Sc_{2,j}$ which are contained in $R$. We now have
\begin{align*}
  \sum_{\substack{Q \in \Sc_{2,j} \\ Q \subset R}} w(Q) \Big(\frac{w(Q)}{\int_{\R^n} (M1_Q)^q w}\Big)^{\frac{p'}{sr}-1}
  &\overset{\text{(A)}}{\le} \sum_{\substack{Q \in \Sc_{2,j} \\ Q \subset R}} 2^{1-j} \int_{\R^n} (M1_Q)^q w \Big(\frac{2^{1-j} \int_{\R^n} (M1_Q)^q w}{\int_{\R^n} (M1_Q)^q w}\Big)^{\frac{p'}{sr}-1} \\
  &= 2^{1-j + (1-j)\big( \frac{p'}{sr} -1 \big)} \sum_{\substack{Q \in \Sc_{2,j} \\ Q \subset R}} \int_{\R^n} (M1_Q)^q w \\
  &= 2^{(1-j)\frac{p'}{sr}} \sum_{P \in \Sc_{2,j}^*} \, \sum_{\substack{Q\in \Sc_{2,j} \\ Q\subset P}} \int_{\R^n} (M1_Q)^q w \\
  &\overset{\text{(B)}}{\le} 2^{(1-j)\frac{p'}{sr}} \big( [w]_{C_q} + 1 \big) \sum_{P \in \Sc_{2,j}^*} \int_{\R^n} (M1_P)^q w \\
  &\overset{\text{(A)}}{\le} 2^{(1-j)\frac{p'}{sr} + j} \big( [w]_{C_q} + 1 \big) \sum_{P \in \Sc_{2,j}^*} w(P) \\
  &\overset{\text{(C)}}{\le} 2^{(1-j)\frac{p'}{sr} + j} \big( [w]_{C_q} + 1 \big) w(R),
\end{align*}
where we used (A) the definition of the collection $\Sc_{2,j}$, (B) Lemma \ref{summation-sparse-lemma} and (C) the fact that the cubes in $\Sc_{2,j}^*$ are disjoint. We now sum over $j$ and get
$$
  \sum_{\substack{Q \in \Sc_{2} \\ Q \subset R }} a_Q
   = \sum_{j \ge 1} \sum_{\substack{Q \in \Sc_{2,j} \\ Q \subset R }} a_Q
   \le \big( [w]_{C_q} + 1 \big) 2^{\frac{p'}{sr}} \sum_{j\geq 1} 2^{j\big(1-\frac{p'}{sr}\big)} w(R).
$$
Therefore \eqref{claim:carleson} holds with
\[
A := \big( [w]_{C_q} + 1 \big) 2^{\frac{p'}{sr}} \sum_{j\geq 1} 2^{j\big(1-\frac{p'}{sr}\big)} = 2 \frac{\big( [w]_{C_q}+1\big)} {1-2^{1-p'/sr}} \leq \tilde c_p  \big( [w]_{C_q}+1\big).
\]
Putting all of the above together, we proved that for $s = 1 + \frac{\delta}{8p}$ we have
$$
  s'\sum_{Q\in \mathcal S} \langle |f|\rangle_{Q}  \langle |gw|\rangle_{s,Q}|Q|
  \le s' \big( c_{n,p,q} ([w]_{C_q}+1)^2\log ([w]_{C_q}+e) + c_p   \big( [w]_{C_q}+1\big) \big).
$$
The constant $c_{n,p,q}$ is the same constant as in Corollary \ref{Corollary-norm} and thus, we have
$$
  c_{n,p,q} \approx c_n 2^{c_n' \frac{pq}{q-p}}
$$
by \cite[Lemma 5.8]{canto}. In particular, $c_{n,p,q} \to \infty$ as $q \to p$. Since $\delta = \frac{1}{B \max\{[w]_{C_q},1\}}$ where $B = B(n,q)$ (see \cite[Theorem 2.13]{canto}), we have
$$
  s' = \frac{8p}{\delta} + 1 \approx 8p B \max\{[w]_{C_q},1\}.
$$
Hence we see that
$$
  s'\sum_{Q\in \mathcal S} \langle |f|\rangle_{Q}  \langle |gw|\rangle_{s,Q}|Q|
  \le C_{n,p,q} \big([w]_{C_q} + 1 \big)^3 \log\big( [w]_{C_q} + e \big)
$$
for a constant $C_{n,p,q}$ such that $C_{n,p,q} \to \infty$ as $q \to p$.

\section{Sparse domination for rough singular integrals revisited} \label{section:sparse}

Before we prove part II) of Theorems \ref{thm:main_result_rough} and \ref{thm:main_sparse_bounds}, we revisit the sparse domination principle in \cite{condealonsoculiucdiplinioou} and prove a version of it that is more suitable for the case $0 < p < 1$. Let us first consider a Calder\'on--Zygmund operator $T$. It is now well-known (see e.g. \cite{lacey, hytonenroncaltapiola, lerner_pointwise}) that $T$ satisfies a pointwise sparse bound of the type
$$
  Tf(x) \le C_T \sum_{i, Q \in \Sc_i} 1_Q (x) \langle |f| \rangle_Q.
$$
Now, for $0 < p < 1$, we trivially have
$$
  |Tf(x)|^p \le C_T^p \sum_{i, Q \in \Sc_i} 1_Q(x) \langle |f| \rangle_Q^p,
$$
and thus, for $q = 1 + \lambda$ and $w \in C_q$ for any $\lambda > 0$, Corollary \ref{Corollary-norm} gives us
\begin{align*}
  \int_{\R^n} |Tf|^p w &\le C_T^p \sum_{i, Q \in \Sc_i} w(Q) \langle |f| \rangle_Q^p \\
                       &\le C_T^p \sum_{i, Q \in \Sc_i} \langle |f| \rangle_Q^p \int_{\R^n} (M1_Q)^q w \\
                       &\le C_T^p c_{n,p,q} ([w]_{C_q}+1)^2\log ([w]_{C_q}+e) \| Mf\|_{L^p(w)}^p.
\end{align*}
Qualitative version of this result was proven recently as a part of \cite[Theorem 17]{cejasliperezrivera-rios} using different techniques.

To mimic this proof strategy for rough homogeneous singular integrals, we prove the following sparse domination result:

\begin{theorem}
  \label{thm:sparse_domination}
  Suppose that $0 < \theta < 1$ and $1 < s \leq \frac{1}{1-\theta}$. Then there exists a sparse collection $\Sc$ such that
  \begin{equation*}
    |\langle  |T_\Omega f|^\theta, g \rangle| \lesssim  (s')^\theta\|\Omega\|_{L^\infty(\Sb^{n-1})}^\theta \sum_{Q\in \Sc}|Q| \langle |f| \rangle_Q^\theta \langle |g| \rangle_{s,Q}.
  \end{equation*}
\end{theorem}

Our proof is strongly based on techniques used by Lerner in \cite{lerner_weak-type}. For a sublinear operator $T$ and $0 < \theta < 1$, we define
$$
  \Ms_{T}^\theta (f,g)(x) \coloneqq \sup_{Q \ni x} \frac{1}{|Q|} \int_Q |T(f1_{\R^n \setminus 3Q})|^\theta |g| \, dy.
$$
Our main tool is the following variant of \cite[Theorem 3.1]{lerner_weak-type}:

\begin{theorem}
  \label{thm:lerner_maximal}
  Let $1 \le q\le r$, $0< \theta < 1$ and $s \ge 1$. Assume that $T$ is a sublinear operator of weak type $(q,q)$ and $\Ms_T^\theta$ satisfies the following estimate:
$$
    \|\mathscr M_T^\theta(f,g)\|_{L^{\nu,\infty}}\le N \|f\|_{L^r}^\theta \|g\|_{L^s},
$$
  for exponents satisfying the relation
  \begin{align*}
    \frac{1}{\nu} = \frac{\theta}{r} + \frac{1}{s}.
  \end{align*}
  Then for every compactly supported $f \in L^r(\R^n)$ and every $g \in L_{\text{loc}}^s$, there exists a sparse collection of cubes $\Sc$ such that
$$
    \langle |Tf|^\theta, |g| \rangle \le C_{T,N} \sum_{Q \in \Sc} |Q| \langle |f| \rangle_{r,Q}^\theta \langle |g| \rangle_{s,Q},
$$
  where
$
    C_{T,N} \coloneqq c_n \big(\|T\|_{L^q\rightarrow L^{q,\infty}}^\theta+N \big).
$
\end{theorem}

\begin{proof}
  The proof is essentially the same as the proof of \cite[Theorem 3.1]{lerner_weak-type}. The only difference is the definition of the sets $E_1$ and $E_2$: the first set is the same, namely
$$
    E_1 = \{x \in Q_0 \colon |T(f1_{3Q_0})| > A \langle |f|\rangle_{q, 3Q_0} \},
$$
  and we define the second set as
 $$
    E_2 = \{x \in Q_0 \colon \Ms_{T, Q_0}^\theta(f,g)(x) > B \langle |f|\rangle_{r, 3Q_0}^\theta \langle |g|\rangle_{s, Q_0}\}.
$$
  The rest of the proof works as it is with the the obvious changes.
\end{proof}

With the help of Theorem \ref{thm:lerner_maximal}, the proof of Theorem \ref{thm:sparse_domination} is fairly straightforward.

\begin{proof}[Proof of Theorem \ref{thm:sparse_domination}]
  Let $T_\Omega$ be a rough homogeneous singular integral. We want to apply Theorem \ref{thm:lerner_maximal} with $q = 1 = r$. Let $1 < s \le \frac{1}{1-\theta}$. Since $T_\Omega$ is of weak-type $(1,1)$ by \cite{seeger}, we only need to check the bound for $\Ms_{T_\Omega}^\theta$. To be more precise, we need to show that
  \begin{equation}
    \label{pf:to_check}
    \|\Ms_{T_\Omega}^\theta(f,g)\|_{L^{\nu,\infty}} \le N \|f\|_{L^1}^\theta \|g\|_{L^s},
  \end{equation}
  where
$
    \frac{1}{\nu} = \theta + \frac{1}{s}.
$
  Let us define an auxiliary operator $\Ns_{p, T_\Omega}^\theta$ by setting
 $$
    \Ns_{p, T_\Omega}^\theta f(x) = \sup_{Q \ni x} \Big(\frac{1}{|Q|} \int_Q |T_\Omega(f1_{\R^n \setminus 3Q})|^{p\theta} dy \Big)^{\frac{1}{p}}.
$$
  Notice that we have $\Ns_{p, T_\Omega}^\theta f(x) = \big(\Ns_{p\theta, T_\Omega}^1 f(x) \big)^\theta$. By H\"older's inequality, we have the pointwise bound
  \begin{align*}
    \Ms_{T_\Omega}^\theta (f,g)(x)
    &\le \sup_{Q \ni x} \Big( \int_Q \big| T_\Omega ( f1_{\R^n \setminus 3Q}) \big|^{s' \theta} \Big)^{\frac{1}{s'}} \Big( \int_Q |g|^s \Big)^{\frac{1}{s}} \\
    &\le \Ns_{s', T_\Omega}^\theta f(x) M_s g(x)
    = \big( \Ns_{s'\theta, T_\Omega}^1 f(x) \big)^\theta M_s g(x).
  \end{align*}
 Now, combining this pointwise bound with H\"older's inequality for weak spaces (see e.g. \cite[Ex. 1.1.15]{grafakos}), the straightforward estimate $\big\| \big( \Ns_{s'\theta, T_\Omega}^1 f \big)^\theta \big\|_{L^{\frac{1}{\theta},\infty}} = \big\| \Ns_{s'\theta, T_\Omega}^1 f \big\|_{L^{1,\infty}}^\theta$ and the weak type $(s,s)$ of $M_s$ (see Section \ref{section:maximal_and_singular}) we get
  \begin{align*}
    \|\Ms_{T_\Omega}^\theta(f,g)\|_{L^{\nu,\infty}}
    &\le \nu^{-\frac{1}{\nu}} \theta^{-\theta} s^{\frac{1}{s}} \big\| \big( \Ns_{s'\theta, T_\Omega}^1 f \big)^\theta \big\|_{L^{\frac{1}{\theta},\infty}} \|M_s g\|_{L^{s,\infty}}\\
    &\lesssim \nu^{-\frac{1}{\nu}} \theta^{-\theta} s^{\frac{1}{s}} \| \Ns_{s'\theta, T_\Omega}^1 f \|_{L^{1,\infty}}^\theta \|g\|_{L^s}.
  \end{align*}
  By \cite[Theorem 1.1, Lemma 3.3]{lerner_weak-type}, we know that
  $$
    \| \Ns_{s'\theta, T_\Omega}^1 f \|_{L^{1,\infty}} \lesssim s'\theta \|\Omega\|_{L^\infty(\Sb^{n-1})} \|f\|_{L^1},
$$
  provided that $1 \le s'\theta < \infty$ which is equivalent to $1 < s \le \frac{1}{1-\theta}$. Therefore, we have
  \begin{align*}
    \|\Ms_{T_\Omega}^\theta(f,g)\|_{L^{\nu,\infty}}
    &\lesssim \nu^{-\frac{1}{\nu}} \theta^{-\theta} s^{\frac{1}{s}} (s'\theta)^\theta \|\Omega\|_{L^\infty(\Sb^{n-1})}^\theta \|f\|_{L^1}^\theta \|g\|_{L^s} \\
    &\lesssim (s')^\theta\|\Omega\|_{L^\infty(\Sb^{n-1})}^\theta \|f\|_{L^1}^\theta \|g\|_{L^s},
  \end{align*}
since $\theta < 1 < s$, $\nu = s/(\theta s +1)$ and
$$
 \nu^{-\frac{1}{\nu}} \theta^{-\theta} s^{\frac{1}{s}} (s'\theta)^\theta  = s^{-\theta} (s')^\theta (s\theta+1)^{\theta + \frac 1s} \lesssim s^{\frac 1s} (s')^\theta \lesssim (s')^\theta.
$$
Thus, \eqref{pf:to_check} holds for $N = c_n (s')^\theta\|\Omega\|_{L^\infty(\Sb^{n-1})}^\theta$. Since $\|T_\Omega\|_{L^1 \to L^{1,\infty}} \lesssim \|\Omega\|_{L^\infty(\Sb^{n-1})}$ by \cite{seeger}, we can apply \ref{thm:lerner_maximal}, which finishes the proof.
\end{proof}

\section{Proof of part II) of Theorems \ref{thm:main_result_rough} and \ref{thm:main_sparse_bounds}} \label{section:proof_part_ii}

Firstly, we deduce part II) of Theorem \ref{thm:main_result_rough} from the sparse domination presented in Theorem \ref{thm:sparse_domination} and the bound for the sparse form from Theorem \ref{thm:main_sparse_bounds}. Let $0<p\le 1$, we have
\[
\| T_\Omega f\|_{L^p(w)} = \| |T_\Omega f|^p \|_{L^{1}(w)}^\frac{1}{p} =   | \langle |T_\Omega f|^p, w \rangle |^\frac 1p.
\]
Now, we use Theorem \ref{thm:sparse_domination} to dominate the term $|\langle |T_\Omega f|^p, w \rangle |$, and apply part II) of Theorem~\ref{thm:main_sparse_bounds}. We get
\begin{align*}
| \langle |T_\Omega f|^p, w \rangle |^\frac 1p
	& \lesssim \| \Omega\|_{L^\infty}   s' \Big( \sum_{Q\in \Sc} |Q| \langle f \rangle_Q^p \langle w \rangle_{s,Q} \Big)^\frac1p \\
	& \le  c_{n,p,q}  \| \Omega\|_{L^\infty}  \big( [w]_{C_q} + 1 \big)^{1+\frac2p} \log^{\frac1p} \big( [w]_{C_q} + e\big) \| Mf\|_{L^p(w)}.
\end{align*}

We now turn to the proof of part II) of Theorem \ref{thm:main_sparse_bounds}. Suppose that $0 < p \leq 1$, $w \in C_q$ for some $q > 1$ and $\Sc$ is a sparse collection. We want to show that there exists $1 < s < \min\{2,\frac 1{1-p}\}$ such that
$$
  (s')^p \sum_{Q\in \Sc} |Q| \langle |f| \rangle_Q^p \langle w \rangle_{s,Q} \le c_{w,p,q,n} \|Mf\|_{L^p(w)}^p .
$$
We choose $s = 1+ p\delta$, where $\delta$ is the Reverse H\"older exponent from Theorem \ref{thm:rhi_cp}. Hence $s' \lesssim ([w]_{C_q}+1)/p$ and we have
\begin{align*}
(s')^p \sum_{Q\in \Sc} |Q| \langle |f| \rangle_Q^p \langle w \rangle_{s,Q}
&\lesssim  \Big(\frac{[w]_{C_q}+1}p\Big)^p\sum_{Q\in \Sc}  \langle |f| \rangle_Q^p \int \big(M1_Q\big)^q w\\
&\lesssim p^{-p} \big( [w]_{C_q} + 1 \big)^{p+2} \log  \big( [w]_{C_q} + e\big)\|Mf\|_{L^p(w)}^p,
\end{align*}where we have used Corollary \ref{Corollary-norm} in the last step. The implicit constant $c_{n,p,q}$ satisfies $c_{n,p,q} \to \infty$ as $q \to p$ by the same arguments as in the end of Section \ref{section:proof_part_i}. This completes the proof of Theorem \ref{thm:main_sparse_bounds}.

\section{Reverse H\"older and weak self-improving properties of $C_p$} \label{section:RH}

It is well-known that $A_p$ weights are self-improving: if $w \in A_p$, then there exists $\eps > 0$ such that $w \in A_{p-\eps}$ \cite[Lemma 2]{coifmanfefferman}. Since this is a particularly convenient property in many proofs, it would be desirable if $C_p$ weights had a similar property, i.e. for every $w \in C_p$ there existed $\eps > 0$ such that $w \in C_{p+\eps}$. In particular, this property together with Sawyer's results would prove Muckenhoupt's conjecture. Unfortunately, this is not true due to an example by Kahanp\"a\"a and Mejlbro \cite[Theorem 11]{kahanpaamejlbro}. We discuss their counterexample and its generalizations in detail in Section \ref{section:kahanpaamejlbro}.

The failure of this self-improving property raises natural questions about weaker self-improving properties of $C_p$ weights. For example, although the well-known self-improving property of classical Reverse H\"older weights \cite[Lemma 3]{gehring} fails in spaces of homogeneous type \cite[Section 7]{andersonhytonentapiola}, the weights are still self-improving in a weak sense even in this more general setting \cite[Section 6]{andersonhytonentapiola} (see also \cite[Theorem 3.3]{zatorska-goldstein}). Although we show in Section \ref{section:weak_cp} that weakening the definition of $C_p$ in an obvious way does not actually change the structure of the corresponding weight class, various self-improvement and Reverse H\"older questions remain open. In particular:

\begin{openproblem}
  \label{problem:rhi}
  Suppose that $w \in C_p$ for some $1 < p < \infty$ and let $\delta$ be the Reverse H\"older parameter from Theorem \ref{thm:rhi_cp}. Does there exist $c_w > 1$ such that
$$
    \Big( \avgint_Q w^{c(1+\delta)}\Big)^{\frac{1}{c(1+\delta)}} \lesssim \frac{1}{|Q|} \int_{\R^n} M(1_Q)^p w
$$
  for every cube $Q$ and every $1 < c \le c_w$?
\end{openproblem}

In this section, we record two observations related to Problem \ref{problem:rhi}. First, we prove the following analogue to the well-known $A_\infty$ result ``$w \in A_\infty \Rightarrow w^{1+\eps} \in A_\infty$ for small $\eps$'' (see e.g. \cite[Corollary 3.17]{hytonenroncaltapiola}):

\begin{proposition}
  \label{prop:power_weight}
  Let $w \in C_p$ for some $1 < p < \infty$. Then there exists $\eps_0 > 0$ such that $w^{1+\eps} \in C_p$ for every $0 < \eps \leq \eps_0$.
\end{proposition}

\begin{proof}
  Let $\delta$ be the Reverse H\"older parameter from Theorem \ref{thm:rhi_cp} and set $\eps_0 = \frac{\delta}{2}$. Then, for $s = 1 + \frac{\delta}{2+\delta}$, we have $s(1+\eps_0) = 1+\delta$. Thus, we get
  \begin{align*}
    \Big( \avgint_Q w^{(1+\eps_0)s} \Big)^\frac{1}{s}
    &\le \Big( \frac{1}{|Q|} \int_{\R^n} (M1_Q)^p w \Big)^\frac{1+\delta}{s}
     =  \Big( \frac{1}{|Q|} \int_{\R^n} (M1_Q)^p w\Big)^{1+\eps_0}\\
    &\le \Big( \frac{1}{|Q|} \int_{\R^n} (M1_Q)^p\Big)^\frac{1+\eps_0}{1+\frac{1}{\eps_0}} \Big( \frac{1}{|Q|} \int_{\R^n} (M1_Q)^p w^{1+\eps_0}\Big) \\
    &\le( c_n \, p')^{\eps_0} \cdot \frac{1}{|Q|} \int_{\R^n} (M1_Q)^p w^{1+\eps_0},
  \end{align*}
  where we used first Theorem \ref{thm:rhi_cp}, then the standard H\"older's inequality and finally the $L^p$-boundedness of the Hardy--Littlewood maximal operator. Thus, the weight $w^{1+\eps_0}$ satisfies a Reverse H\"older inequality in the sense of Theorem \ref{thm:rhi_cp} and therefore $w^{1+\eps_0} \in C_p$.

  The fact that now also $w^{1+\eps} \in C_p$ for every $0 < \eps \le \eps_0$ follows easily from H\"older's inequality: for every cube $Q$ we have $\langle w^{1+\eps} \rangle_Q^{\frac{1}{1+\eps}} \le \langle w^{1+\eps_0} \rangle_Q^{\frac{1}{1+\eps_0}}$.
\end{proof}

In the light of Proposition \ref{prop:power_weight}, answering the following question would solve Problem \ref{problem:rhi}:
\begin{openproblem}
  \label{problem:rhi2}
  Suppose that $w \in C_p$ for some $1 < p < \infty$. Do there exist $\eps_0 > 0$ and $C \ge 1$ such that
\begin{equation}
    \label{question:rh}
   \Big( \frac{1}{|Q|} \int_{\R^n} (M1_Q)^p w^{1+\eps} \Big)^\frac{1}{1+\eps}
    \le C \frac{1}{|Q|} \int_{\R^n} (M1_Q)^p w
  \end{equation}
  for every cube $Q$ and every $0 < \eps \le \eps_0$?
\end{openproblem}
As a consequence of Proposition \ref{prop:power_weight} we get something slightly worse than \eqref{question:rh}:

\begin{corollary}
  \label{tail-self-improv}
  Suppose $w\in C_p$ for some $1 \le p < \infty$ and let $\delta_0$ be the Reverse H\"older exponent from Theorem \ref{thm:rhi_cp}. Then for every $0 < \delta \le \delta_0$ and every cube $Q$ we have
$$
  \Big( \frac{1}{|Q|} \int_{\R^n} (M1_Q)^p w^{1+\delta} \Big)^\frac{1}{1+\delta}
    \le C_{n,p,\delta} \frac{1}{|Q|} \int_{\R^n} (M1_Q)^{\frac{p+\delta}{1+\delta}} w,
$$
\end{corollary}
Since the proof of Corollary \ref{tail-self-improv} is a fairly technical computation, we formulate explicitly the following well-known embedding property of $\ell^p$ spaces:

\begin{lemma}
  \label{Embedding}
  Let $0 < \alpha < \beta < \infty$. Then, for positive numbers $a_n$, $n \in \N$, we have
$$
   \Big( \sum_n a_n^\beta \Big)^\frac{1}{\beta}
    \le \Big( \sum_n a_n^\alpha \Big)^\frac{1}{\alpha}.
$$
\end{lemma}

\begin{proof}[Proof of Corollary \ref{tail-self-improv}]
  We argue by discretizing the tail. By \cite[Lemma 3.2]{canto}, we have
$$
    \frac{1}{|Q|}\int_{\R^n} (M1_Q)^p w \approx_{n,p} \sum_{k=0}^\infty 2^{-n(p-1)k} \avgint_{2^kQ} w,
$$
  for $1 \le p < \infty$ and any weight $w$. The implicit constants do not blow up when $p$ tends to $1$, but they do blow up when $p \rightarrow \infty$. We get
  \begin{align*}
    \frac{1}{|Q|} \int_{\R^n}(M1_Q)^p w^{1+\delta}
    &\overset{\text{(A)}}{\approx}_{n,p} \sum_{k=0}^\infty 2^{-n(p-1)k} \avgint_{2^kQ} w^{1+\delta} \\
    &\overset{\text{(B)}}{\lesssim} \sum_{k=0}^\infty 2^{-n(p-1)k}\Big(\frac{1}{|2^{k}Q|} \int_{\R^n}(M1_{2^kQ})^p w\Big)^{1+\delta} \\
    &\overset{\text{(A)}}{\lesssim}_{n,p,\delta} \sum_{k=0}^\infty 2^{-n(p-1)k} \Big( \sum_{j=0}^\infty 2^{-n(p-1)j}  \avgint_{2^{j+k}Q} w\Big)^{1+\delta}\\
    &\overset{\text{(C)}}{\le}\Big( \sum_{k,j=0}^\infty 2^{-n(p-1) \frac{k}{1+\delta}} 2^{-n(p-1)j}  \avgint_{2^{j+k}Q} w\Big)^{1+\delta} \\
    &= \Big( \sum_{m=0}^\infty \Big( \sum_{i=0}^m 2^{-n(p-1)\big( \frac{i}{1+\delta}+(m-i)\big)} \Big) \avgint_{2^mQ}w\Big)^{1+\delta} \\
    &\overset{\text{(D)}}{\lesssim} \Big( \sum_{m=0}^\infty  2^{-n(p-1) \frac{m}{1+\delta}} \avgint_{2^mQ}w\Big)^{1+\delta} \\
    &=\Big( \sum_{m=0}^\infty  2^{-n\Big( \frac{p+\delta}{1+\delta}-1 \Big) m} \avgint_{2^mQ}w\Big)^{1+\delta} \\
    &\overset{\text{(A)}}{\approx}_{n,p,\delta}\Big( \frac{1}{|Q|} \int_{\R^n} (M1_Q)^{\frac{p+\delta}{1+\delta}} w\Big)^{1+\delta},
  \end{align*}
  where we (A) used the discretization, (B) used the Reverse H\"older inequality, (C) applied Lemma~\ref{Embedding} with $\alpha = \frac{1}{1 + \delta}$ and $\beta = 1$, and (D) calculated the geometric sum and made obvious estimates.
\end{proof}

\section{On weak $C_p$ and dyadic $C_p$}
\label{section:weak_cp}

When we compare the characterizations of $A_\infty$ \eqref{defin:a_infty} and $C_p$ \eqref{cond:c_p}, it is obvious that $A_\infty \subset C_p$ for every $p$. However, $A_\infty$ weights are not good representatives of $C_p$ weights because the $C_p$ classes are much bigger than the $A_\infty$ class. For example, $A_\infty$ weights are always doubling and they cannot vanish in a set of positive measure whereas $C_p$ weights can grow arbitrarily fast and their supports can contain holes of infinite measure. Thus, the structure of a general $C_p$ weight can be very messy.

In this section, we consider some examples and properties related to $C_p$ weights. We also introduce weak and dyadic $C_p$ weights as an analogy to weak and dyadic $A_\infty$ weights. Although these new classes of weights seem like they are larger than $C_p$, this is not the case: weak and dyadic $C_p$ weights are just $C_p$ weights.

We start by proving an elementary lemma for the Hardy--Littlewood maximal operator that we already used in the previous sections:

\begin{lemma}
  \label{lemma:hl_bound}
  Let $Q_0 \subset \R^n$ be a cube and $E_0 \subset Q_0$ a measurable subset such that $|E_0| \ge \eta |Q_0|$
  for some $0 < \eta \le 1$. Then there exists a structural constant $C_n$ such that
$$
    M(1_{Q_0})(x) \le \frac{C_n}{\eta} M(1_{E_0})(x)
$$
  for almost every $x \in \R^n$.
\end{lemma}

\begin{proof}
  Let $Q(x,r)$ be the cube with center point $x$ and side length $r$. There exists a structural constant $c_n \ge 1$ such that
$$
    E_0 \subset Q_0 \subset Q(x, c_n(\dist(x,Q_0) + \ell(Q_0))).
$$
  The proof now consists of two cases:
  \begin{enumerate}
    \item[1)] Suppose that $\dist(x,Q_0) \le \ell(Q_0)$. Then $Q_0 \subset Q(x,2c_n\ell(Q_0)) \eqqcolon Q_x$ and $|Q_0| \approx |Q_x|$. Thus,
$$
      M(1_{E_0})(x) \ge \frac{|E_0 \cap Q_x|}{|Q_x|} \approx \frac{|E_0|}{|Q_0|} \ge \frac{1}{\eta} \ge \frac{1}{\eta} M(1_{Q_0})(x).
$$

    \item[2)] Suppose that $\dist(x,Q_0) > \ell(Q_0)$. Then
    \begin{align*}
      &M(1_{Q_0})(x) = \sup_{r > \dist(x,Q_0)} \frac{|Q_0 \cap Q(x,r)|}{|Q(x,r)|}
                    \le \sup_{r > \dist(x,Q_0)} \frac{ c_n'|Q_0|}{|Q(x,2c_n r)|} \\
                    &\le \sup_{r > \dist(x,Q_0)} \frac{c_n'}{\eta} \frac{|E_0|}{|Q(x,2c_n r)|}
                    = \sup_{r > \dist(x,Q_0)} \frac{c_n'}{\eta} \frac{|E_0 \cap Q(x,2c_n r)|}{|Q(x,2c_n r)|}
                    \le \frac{c_n'}{\eta} M(1_{E_0})(x).  \qedhere
    \end{align*}
  \end{enumerate}
\end{proof}

\subsection{Weak $A_\infty$ weights}

Let us recall the definition of the weak $A_\infty$ classes. The Fujii--Wilson type characterization of these weights was studied in detail in \cite{andersonhytonentapiola} but earlier they have appeared in other forms in the study of e.g. weighted norm inequalities \cite{sawyer_two_weight} and elliptic partial differential equations and quantitative rectifiability; see e.g. \cite{hofmannlemartellnystrom} and references therein.

\begin{defin}
  Suppose that $\gamma \ge 1$. We say that a weight $w$ belongs to the \emph{$\gamma$-weak $A_\infty$ class $A_\infty^\gamma$} if there
  exist positive constants $C, \delta > 0$ such that
  \begin{equation}
    \label{defin:weak_a_infty}
    w(E) \le C \left( \frac{|E|}{|Q|} \right)^\delta w(\gamma Q)
  \end{equation}
  for any cube $Q$ and any measurable subset $E \subset Q$, where $\gamma Q$ is the cube of side length $\gamma \ell(Q)$ with the same center point as $Q$.
\end{defin}

We denote $A_\infty^\text{weak} \coloneqq \bigcup_{\gamma \ge 1} A_\infty^\gamma$. It was shown in \cite{andersonhytonentapiola} that this definition does not give us a continuum of different weak $A_\infty$ classes but the dilation parameter $\gamma$ is irrelevant for the structure of the class as long as $\gamma > 1$:

\begin{theorem}[{\cite{andersonhytonentapiola}}]
  We have
  \begin{enumerate}
    \item[i)] $A_\infty \subsetneq A_\infty^\gamma$ for every $\gamma > 1$;
    \item[ii)] $A_\infty^\gamma = A_\infty^\text{weak}$ for every $\gamma > 1$;
    \item[iii)] $w \in A_\infty^\text{weak}$ if and only if for every $\lambda > 1$
                there exists a constant $[w]_{A_\infty^\lambda}$ such that
             $$
                  \int_Q M(1_Q w) \le [w]_{A_\infty^\lambda} w(\lambda Q)
              $$
                for every cube $Q$.
  \end{enumerate}
\end{theorem}

\subsection{Examples and some properties of $C_p$ weights}

Let us then gather some known results from the literature and consider
some other examples and properties of $C_p$ weights.
\begin{enumerate}
  \item[i)] From $A_p$ theory, \eqref{defin:weak_a_infty}, Lemma \ref{lemma:hl_bound},
            \cite{andersonhytonentapiola} and Theorem \ref{thm:kahanpaamejlbro}, it follows that for $1 < p < q < \infty$ we have
         $$
              A_1 \subsetneq A_p \subsetneq A_q \subsetneq A_\infty \subsetneq A_\infty^\text{weak} \subsetneq C_q \subsetneq C_p \subsetneq C_1.
         $$

  \item[ii)] If follows easily from the argument in
             \cite[Example 3.2]{andersonhytonentapiola} that $A_\infty^\text{weak}$
             contains all non-negative functions that are monotonic in each variable.
             By i), all these functions are also contained in $C_p$ for every $p$. In particular, $C_p$ weights are generally non-doubling.

  \item[iii)] If $w \in C_p$ is a doubling weight such that $w(\widetilde{Q}) \le 2^p w(Q)$,
              where $\widetilde{Q}$ is the concentric dilation of $Q$ with $\ell(\widetilde{Q}) = 2 \ell(Q)$, then $w \in A_\infty$ \cite[Section 7]{buckley}.

  \item[iv)] If $w \in A_\infty$, then $w 1_{[0,\infty)} \in C_p$ for every
              $1 \le p < \infty$ \cite{muckenhoupt_norm_inequalities}.

  \item[v)] More generally, if $w \in A_\infty$ and $g$ is a \emph{convexely contoured}
             weight (i.e. a weight such that $\{x \in \R^n \colon g(x) < \alpha\}$ is a convex set for every $\alpha \ge 0$), then $wg \in C_p$ for every $1 \le p < \infty$ \cite[Proposition 7.3]{buckley}.

  \item[vi)] If $w$ is a compactly supported weight, then $w \notin C_p$ for any $p$.
            It is straightforward to prove this. Let us denote $P \coloneqq \text{supp} \, w$. For every $k \in \N$, let $P_k$ be a cube such that $P \subset P_k$ and $|P_k| \ge 2^k |P|$. Now, for $E = P$, we have
$$
              \int_{\R^n} (M1_{P_k})^p w = \int_P (M1_{P_k})^p w = \int_P w = w(P) \in (0,\infty)
$$
            for every $k$ since $w$ is locally integrable. However,
$$
             \Big( \frac{|E|}{|P_k|} \Big)^\eps \le\Big( \frac{|P|}{2^k|P|} \Big)^\eps \searrow 0 \quad \text{ as } k \nearrow \infty
$$
            for every $\eps > 0$. Thus, there do not exist constants $C$ and $\eps$ such that \eqref{cond:c_p} holds for every cube $Q$. This argument also proves that if $w\in C_p$, then $w \notin L^1(\R^n)$.

  \item[vii)] Even though $C_p$ weights cannot have compact support, their support can have arbitrarily small measure. Indeed, suppose that $w \in A_\infty$ and $P = \bigcup_{k=1}^\infty [10^k, 10^k + \tfrac{1}{2^k}]$. Then $|P| = 1$ but $P$ is unbounded. We set $v(x) \coloneqq w(x) 1_P(x)$.
             \begin{enumerate}
               \item[$\bu$] If $w(x) = x^4$, then $\int_{\R} M(1_Q)^2 v = \infty$ for every cube $Q$ and thus, $v \in C_2$.

               \item[$\bu$] If $w(x) = 1$, then $w$ is integrable and, by vi), $w \not \in C_p$ for any $p$.
             \end{enumerate}
  \item[viii)] Suppose that $w$ is a weight such that $w(x) \ge \alpha > 0$ for every
              $x \in \R^n \setminus A$, where $A$ is a bounded set. Since $M(1_Q) \notin L^1(dx)$ for any cube $Q$, we have
$$
                \int_{\R^n} M(1_Q) w \ge \alpha \int_{\R^n \setminus A} M(1_Q) = \infty
$$
              and thus, $w \in C_1$.
\end{enumerate}

\subsection{Weak $C_p$ and dyadic $C_p$}

Let us then consider two generalizations of the $C_p$ class. Suppose that $\gamma \ge 1$. We write
\begin{enumerate}
  \item[i)] $w \in C_p^\Ds$ if we $w$ satisfies condition \eqref{cond:c_p} for
            all $Q \in \Ds$ instead of all cubes;
  \item[ii)] $w \in C_p^\gamma$, if $w$ satisfies condition \eqref{cond:c_p}  for
             $1_{\gamma Q}$ instead of $1_Q$, and all cubes $Q$;
  \item[iii)] $w \in C_p^\text{weak}$ if $w \in \bigcup_{\alpha \ge 1} C_p^\alpha$.
\end{enumerate}
We also define $A_\infty^\Ds$ similarly as $C_p^\Ds$.

Usually, these types of generalizations genuinely weaken the objects in question. For example, in the case of $A_\infty$, we already saw that $A_\infty$ is a proper subset of $A_\infty^\text{weak}$, and since $1_{[0,\infty)} \in A_\infty^\Ds$, we also have $A_\infty \subsetneq A_\infty^\Ds$. However, because of the non-local nature of the $C_p$ condition, these generalizations for $C_p$ classes just end up giving us back $C_p$:

\begin{proposition}
  \label{proposition:c_p-structure}
  We have $C_p = C_p^\Ds = C_p^\gamma = C_p^\text{weak}$ for every $\gamma \ge 1$.
\end{proposition}

\begin{proof}
  The inclusions
$$
    C_p \subset C_p^\Ds \quad \text{ and } \quad C_p \subset C_p^\gamma \subset C_p^\text{weak}
$$
  are obvious and
$$
    C_p \supset C_p^\gamma \supset C_p^\text{weak}
$$
  follow from Lemma \ref{lemma:hl_bound}. Thus, we only need to show that $C_p^\Ds \subset C_p$.

  Suppose that $w \in C_p^\Ds$ and let $Q \subset \R^n$ be any cube and $E \subset Q$ a measurable set. There exists $2^n$ dyadic cubes $Q_i \in \Ds$ and a uniformly bounded constant $\alpha \ge 1$ such that
  \begin{enumerate}
    \item[1)] the cubes $Q_i$ are pairwise disjoint,
    \item[2)] $\ell(Q_i) \approx \ell(Q)$,
    \item[3)] $Q \subset \bigcup_i Q_i \subset \alpha Q$.
  \end{enumerate}
  Applying the $C_p^\Ds$ property to the sets $Q_i \cap E$ and Lemma \ref{lemma:hl_bound} to $M(1_{\alpha Q})$ gives us
  \begin{align*}
    w(E) = \sum_i w(E \cap Q_i)
         &\le C \sum_i \Big( \frac{|E \cap Q_i|}{|Q_i|}\Big)^\eps \int_{\R^n} (M1_{Q_i})^p w \\
         &\lesssim C \sum_i \Big( \frac{|E|}{|Q|} \Big)^\eps \int_{\R^n} M(1_{\alpha Q})^p w \\
         &\lesssim C 2^n \Big( \frac{|E|}{|Q|} \Big)^\eps \int_{\R^n} (M1_{Q})^p w. \qedhere
  \end{align*}
\end{proof}

\section{$\widetilde{C}_p$ and Kahanp\"a\"a--Mejlbro counterexample revisited}
\label{section:kahanpaamejlbro}

This last section is devoted to the counterexample constructed by Kahanp\"a\"a and Mejlbro in \cite{kahanpaamejlbro}
and the $C_\psi$ classes introduced by Lerner in \cite{lerner_remarks}. These classes are generalizations of $C_p$ classes that depend on a Young function $\psi$ instead of $p$. Because of the limited availability of \cite{kahanpaamejlbro}, and for convenience of the reader, we give a self-contained description of their counterexample.

We give a detailed proof of the failure of the self-improving properties of $C_p$ classes and generalize this also to the context of $C_\psi$ for a carefully chosen $\psi$. Although we use many central ideas of Kahanp\"a\"a and Mejlbro, the proof we present here is different from the one given in  \cite{kahanpaamejlbro}. In particular, we avoid using the explicit Hilbert transform estimates that had a key role in \cite{kahanpaamejlbro} and our techniques allow us to consider dimensions higher than $1$.

Let us start by recalling the central results and objects.

\subsection{The Kahanp\"a\"a--Mejlbro weights}

As we mentioned earlier, Muckenhoupt's conjecture would be trivially true if every $C_p$ weight was self-improving
with respect to $p$. Unfortunately, this is not true due to a construction by Kahanp\"a\"a and Mejlbro.
For every $k \in \Z$, let us denote
$$
  I_k \coloneqq [4k-3, 4k-1] \quad \text{ and } \quad \Omega_k \coloneqq \Big[4k - \frac{1}{2}\ell_k, 4k + \frac{1}{2} \ell_k \Big],
$$
where $\ell_k \in (0,1]$ are numbers such that $\inf_{k \in \Z} \ell_k = 0$. Let also
$h_k$ be numbers such that $0 < h_k < N$ for every $k  \in \Z$ and some universal constant $N$.
We define the weight $w$ as
\begin{equation}
  \label{THEWEIGHT}
  w = \sum_{k \in \Z} 1_{I_k} + \sum_{k \in \Z} h_k 1_{\Omega_k}.
\end{equation}

We note that in \cite{kahanpaamejlbro}, the sum in the definition of $w$ was indexed as
$k \geq 0$. Here we write $k \in \Z$ because of symmetry and because this way
it is easier to generalize the constructions to higher dimensions.

\begin{theorem}[{\cite[Theorem 11, Proposition 12]{kahanpaamejlbro}}]
  \label{thm:kahanpaamejlbro}
  Let $p>1$. For suitable choices of the numbers $h_k$, the weight $w$ satisfies $w \in C_p$ and $w \notin C_{p+\eps}$ for any $\eps > 0$.
  In particular,
  \begin{align}
    \label{failure:self_improving} C_p \setminus \bigcup_{q > p} C_q \neq \emptyset.
  \end{align}
\end{theorem}
The property \eqref{failure:self_improving} can also be seen as a corollary of Theorem \ref{thm:strict_inclusions} a).

\subsection{The $C_\psi$ classes of Lerner}

The classes $C_\psi$ were introduced by Lerner in \cite{lerner_remarks} as intermediate classes between
$C_p$ and $C_q$ for $q > p \ge 1$ and a new way to attack Muckenhoupt's conjecture. To be more precise,
we define generalizations of $C_p$ classes that depend on a Young function $\psi$ instead of $p$.
As we will see, the choice of the function $\psi$ affects the structure of the class in a significant way.

Let $\psi$ be a function defined on $[0,1]$. We denote $w \in C_\psi$ if there exist constants
$C_w, \varepsilon_w > 0$ such that for every cube $Q$ and measurable $E\subset Q$ we have
\begin{equation}
  \label{C-psi}
  w(E) \le C_w \Big( \frac{|E|}{|Q|}\Big)^{\varepsilon_w} \int_{\R^n } \psi \big( M1_Q \big) w.
\end{equation}
Without loss of generality, we may assume that $C_w \ge 1$.

\begin{example}
  If we choose the function $\psi$ in a suitable way, we recover classes that we have considered earlier:
  \begin{enumerate}
    \item[$\bu$] Let $\psi_p(t) = t^p$, for $1 < p < \infty$. Then $C_{\psi_p} = C_p$.
    \item[$\bu$] Let $\psi_\infty = 1_{\{1\}}$. Then we have $\psi_\infty(M1_Q) = 1_Q$ and thus, $C_{\psi_\infty} = A_\infty$.
    \item[$\bu$] Let $0 < a < 1$ and $\psi_a = 1_{[a,1]}$. Then we have $\psi_a(M1_Q) = 1_{C_a Q}$ for some constant $C_a > 1$ and thus, $C_{\psi_a} = A_\infty^\text{weak}$.
  \end{enumerate}
\end{example}

For the rest of the paper, we consider a $C_\psi$ class with a carefully chosen $\psi$:
\begin{defin}
  Let $p > 1$. We set $\widetilde{C}_p \coloneqq C_{\varphi_p}$ for the function $\varphi_p$ such that $\varphi_p(0) = 0$ and
 $$
    \varphi_p(t) = \frac{t^p}{\log^2(1+\frac{1}{t})}, \quad t\in (0,1].
$$
\end{defin}
For notational convenience, we also set $\varphi_p(t) = \varphi_p(1)$ for every $t > 1$.
It is straightforward to check that the function $\varphi_p$ satisfies the following properties:
\begin{enumerate}
  \item $\lim_{t \to 0} \varphi_p(t) = 0$ and $\varphi_p(1) = \tfrac{1}{\log^2 2} > 1$,

  \item both $\varphi_p$ and $t \mapsto t^{-1} \varphi_p(t)$ are increasing functions,

  \item $\varphi_p(2t) \le C \varphi_p(t)$ for some $C > 0$ and
        all $t \ge 0$ (and thus, $\varphi_p(\lambda t) \le C_\lambda \varphi_p(t)$ for any $\lambda > 1$ and $t \ge 0$),

  \item $\int_0^1 \varphi_p(t) \frac{dt}{t^{p+1}} < \infty$.
\end{enumerate}

The key property of $\widetilde{C}_p$ is that $\bigcup_{q > p} C_q \subset \widetilde{C}_p$ and we have
\begin{equation}
  \label{implications:sharp_max}
  w \in \widetilde{C}_p \ \ \implies \ \ \|Mf \|_{L^p(w)} \lesssim \|M^\sharp f\|_{L^p(w)}
                            \implies \ \ w \in C_p,
\end{equation}
where $M^\sharp$ is the sharp maximal operator of Fefferman and Stein \cite{feffermanstein_hp}. The implications \eqref{implications:sharp_max} were first proven by Yabuta \cite[Theorem 1, Theorem 2]{yabuta} for $w \in \bigcup_{q > p} C_q$ and then improved by Lerner \cite[Theorem 6.1]{lerner_remarks} to this form. By \cite[Remark 6.2]{lerner_remarks} and \cite[Subsection 1.5]{cejasliperezrivera-rios}, we know that this result also gives us \eqref{ineq:coifman-fefferman} for e.g. Calder\'on--Zygmund operators and every $w \in \widetilde{C}_p$.

\begin{theorem}[{\cite[Remark 6.2]{lerner_remarks},\cite[Subsection 1.5]{cejasliperezrivera-rios}}]
\label{thm:Cptilde}
In any dimension, we have: If  $w \in \widetilde{C}_p$ then  \eqref{ineq:coifman-fefferman} holds for
Calder\'on--Zygmund operators.
\end{theorem}

\subsection{The Kahanp\"a\"a--Mejlbro weights and $\widetilde{C}_p$}

Since $\varphi_p(t) \lesssim t^p$ for all $t \in [0,1]$, we have $\widetilde  C_p \subset C_p$. On the other hand, since $t^q \lesssim \varphi_p(t)$ for every $q > p$, we have
$C_q \subset \widetilde C_p$ for any $q > p$. Thus, for any $p > 1$, we have
\begin{equation}
  \label{inclusions:strict}
  \bigcup_{q>p} C_q = \bigcup_{\varepsilon >0} C_{p+\varepsilon} \subset \widetilde C_p \subset C_p.
\end{equation}
This raises a natural question: Are these inclusions strict? If the first one is not,
we get a self-improving property for $\widetilde{C}_p$ weights. If the second one is not,
we have solved Muckenhoupt's conjecture. Unfortunately, we will next show that $\widetilde{C}_p \setminus \bigcup_{q > p} C_q \neq \emptyset$ and $C_p \setminus \widetilde{C}_p \neq \emptyset$. This does not prove or disprove Muckenhoupt's conjecture but it is one step closer to understanding the solution.

Our main tool for proving that the inclusions in \eqref{inclusions:strict} are strict in dimension one is the following generalization of Kahanp\"a\"a--Mejlbro techniques:

\begin{theorem}\label{theorem-counterexample}
  Let $1 < p < \infty$ and let $w$ be a weight as in \eqref{THEWEIGHT}.
  \begin{enumerate}
    \item[i)] If $w \in C_p$, then $h_k \lesssim (\ell_k)^{p-1}.$
    \item[ii)] If $h_k = (\ell_k)^{p-1}$, then $w\in C_p$.
    \item[iii)] If $w \in \widetilde C_p$, then $h_k \lesssim \int_0^{\ell_k} \varphi_p(t) \frac{dt}{t^2}.$
    \item[iv)] If $h_k = \frac{\varphi_p(\ell_k)}{\ell_k}$, then $w\in \widetilde C_p$.
  \end{enumerate}
\end{theorem}
In i) and iii) we mean that the inequality holds for all $k$ with implicit constant independent of $k$. One can also prove similar statements as iii) and iv) for the more general class $C_\psi$ assuming that $\psi$ satisfies certain conditions, but for the sake of simplicity we only consider the class $\widetilde C_p$.

 Before giving the proof of Theorem \ref{theorem-counterexample}, we use the theorem to prove the strictness of the inclusions:

\begin{theorem}
  \label{thm:strict_inclusions}
  We have
  \begin{enumerate}
    \item[a)] $C_p \setminus \widetilde C_p \neq \emptyset$,
    \item[b)] $ \cup_{\varepsilon>0} C_{p+\varepsilon} \subsetneq \widetilde C_p$.
  \end{enumerate}
\end{theorem}

\begin{proof}
  We construct weights $w$ of the type \eqref{THEWEIGHT} and then use Theorem \ref{theorem-counterexample} to prove the claims.
  \begin{enumerate}
    \item[a)] Let us set $h_k = (\ell_k)^{p-1}$ for every $k \in \Z$. By part ii) of Theorem
               \ref{theorem-counterexample}, we know that $w \in C_p$. Let us then use part iii) of Theorem \ref{theorem-counterexample} to show that $w \notin \widetilde{C}_p$. It is enough to show that
              $$
                 \inf_{0<t<1} \frac{\int_0^{t} \varphi_p(s) \frac{ds}{s^2}}{t^{p-1}} =0.
        $$
               This can be seen easily by computing the limit as $t \to 0^+$: by L'H\^opital's rule and the Fundamental theorem of calculus, we have
   $$
                 \lim_{t \to 0^+} \frac{\int_0^{t} \varphi_p(s) \frac{ds}{s^2}}{t^{p-1}}
                 = \lim_{t \to 0^+} \frac{\varphi_p(t) t^{-2}}{(p-1) t^{p-2}}
                 = \frac{1}{p-1} \lim_{t \to 0^+} \frac{ \varphi_p(t)}{t^p}
                 = 0.
$$
               Thus, by part iii) of Theorem \ref{theorem-counterexample}, $w \notin \widetilde{C}_p$.

    \item[b)] Let us set
            $$
                h_k = \frac{ \varphi_p(\ell_k)}{\ell_k}= \frac{(\ell_k)^{p-1} }{\log^2(1+\frac{1}{\ell_k})}.
          $$
              for every $k \in \Z$. By part iv) of Theorem \ref{theorem-counterexample}, we know that $w \in \widetilde C_p$. We then use part i) of Theorem \ref{theorem-counterexample} to show that $w\not \in C_{p+\varepsilon}$ for any
              $\varepsilon>0$. To see this, we prove
            $$
                \inf_{0<t<1} \frac{t^{p+\varepsilon-1}}{\varphi_p(t) t^{-1}} =0.
            $$
              As in the previous case, we show this by computing the limit as $t \to 0^+$. We get
              $$
                \lim_{t \to 0^+} \frac{t^{p+\varepsilon-1}}{\varphi_p(t) t^{-1}}
                = \lim_{t \to 0^+} t^\eps \log^2\Big(\frac{1+t}{t}\Big)
                = \lim_{t \to 0^+} t^\eps \big( \log(1+t) - \log(t) \big)^2
                = 0,
           $$
              since $x^\alpha \log(x) \to 0$ as $x \to 0^+$ for any $\alpha > 0$. Thus, by part i) of Theorem \ref{theorem-counterexample}, $w \not \in C_{p+\varepsilon}.$ for any $\eps > 0$. \qedhere
  \end{enumerate}
\end{proof}
From Theorem \ref{thm:Cptilde} we know that $ \widetilde C_p$ is sufficient for \eqref{ineq:coifman-fefferman}, but from Theorem \ref{thm:strict_inclusions} b) there exists a weight $w \in \widetilde C_p \setminus \cup_{\eps>0}C_{p+\eps}$. In particular:

\begin{corollary}
  \label{cor:not_nec}
  The condition $C_{p+\varepsilon}$ is not necessary for \eqref{ineq:coifman-fefferman} to hold for Calder\'on--Zygmund operators.
\end{corollary}

The following counterpart of \cite[Proposition 8]{kahanpaamejlbro} will be useful for us in the proof of Theorem~\ref{theorem-counterexample}.

\begin{lemma}
  \label{lemma-hole}
  Let $p>1$ and $w \in \widetilde{C}_p$. Then there exists a constant $C = C_{\varphi,w} > 0$ such that for any cube $Q$ we have
$$
    %
    \int_{\R^n} \varphi_p (M1_Q) w
    \le C \int_{\R^n \setminus Q} \varphi_p (M1_Q )w.
$$
  The constant $C$ depends on $\varphi_p$ and $w$.
\end{lemma}

\begin{proof}
  Let us fix a cube $Q$ and set $\alpha = (2 \varphi_p(1) C_w)^{\frac{1}{n \eps_w}}$, where $C_w$ and $\eps_w$ are the constants in the definition of $\widetilde{C}_p = C_{\varphi_p}$ \eqref{C-psi}. Notice that $\alpha \ge 1$. Now applying the $\widetilde{C}_p$ condition for $\alpha Q$ and $Q$ gives us
  \begin{align*}
    w(Q)
    &\le C_w \Big(\frac{|Q|}{\alpha^n |Q|}\Big)^{\eps_w} \int_{\R^n} \varphi_p(M1_{\alpha Q}) w = \frac{1}{2\varphi_p(1)} \int_{\R^n} \varphi_p(M1_{\alpha Q}) w \\
    &\le \frac{1}{2} w(Q) + \frac{1}{2\varphi_p(1)} \int_{\R^n \setminus Q} \varphi_p(M 1_{\alpha Q}) w,
  \end{align*}
  since $M1_{\alpha Q} = 1$ on $Q$ and $\varphi_p(1) > 1$. In particular,
$$
    w(Q) \le \frac{1}{\varphi_p(1)} \int_{\R^n \setminus Q} \varphi_p(M1_{\alpha Q})w.
$$
  Thus,
  \begin{align*}
    \int_{\R^n} \varphi_p(M1_Q)w
    &= \varphi_p(1) w(Q) + \int_{\R^n \setminus Q} \varphi_p(M1_Q)w \\
    &\le \int_{\R^n \setminus Q} \varphi_p(M1_{\alpha Q})w + \int_{\R^n \setminus Q} \varphi_p(M1_Q)w \\
    &\overset{\text{(A)}}{\le} \int_{\R^n \setminus Q} \varphi_p(c_\alpha M1_Q)w + \int_{\R^n \setminus Q} \varphi_p(M1_Q)w \\
    &\overset{\text{(B)}}{\le} C_\alpha \int_{\R^n \setminus Q} \varphi_p(M1_Q)w,
  \end{align*}
  where we used (A) Lemma \ref{lemma:hl_bound} and the fact that $\varphi_p$ is increasing, and (B) the doubling property of $\varphi_p$.
\end{proof}

\begin{remark}\label{REMARK-Cp-tail-hole}
If $w\in C_p$, one can prove with almost the same proof as above that
$$
\int_{\R^n} (M1_Q)^pw \leq C \int_{\R^n\setminus Q} (M1_Q)^pw.
$$
\end{remark}


\begin{proof}[Proof of Theorem \ref{theorem-counterexample}]
  Let us fix an interval $I$ and a subset $E \subset I$. We denote $A \coloneqq \bigcup_k I_k$.
  It is straightforward to check that for almost every $x \in A$ and every $r > 0$ we have
  \begin{equation}
    \label{prop:homogeneity}
    |A \cap (x-r,x+r)| \ge c_A r,
  \end{equation}
  for a uniformly bounded constant $c_A > 0$.
 \begin{enumerate}
    \item[i)]  Suppose that $w \in C_p$. Notice that by the definition of the weight $w$,
              we have $h_k \ell_k = w(\Omega_k)$. To simplify the notation, we only consider the case $k = 0$ and denote $h \coloneqq h_0$, $\ell \coloneqq \ell_k$ and $\Omega_0 \coloneqq \Omega$.
              Now applying the $C_p$ condition for the set $\Omega = [-\tfrac{1}{2}\ell, \tfrac{1}{2} \ell]$ gives us
              \begin{align*}
                h \ell
                = w(\Omega)
                \le    \int_{\R} (M1_{\Omega})^p w
                &\overset{\text{(A)}}{\le} C \int_{\R \setminus \Omega} (M1_{\Omega})^p w
                = C \int_{|x| > \frac{\ell}{2}} (M1_{\Omega}(x))^p w(x) \,dx \\
                &\overset{\text{(B)}}{=} C \int_{|x| > 1}\Big( \sup_{J \ni x} \frac{|\Omega \cap J|}{|J|} \Big)^p w(x) \,dx \\
                &\overset{\text{(C)}}{\le} C_p \int_{|x|>1} \Big( \frac{|\Omega|}{|x|}\Big)^p \, dx
                \le C_p \ell ^p \int_{|x|>1} |x|^{-p} \, dx
                \le C_p \ell^p,
              \end{align*}
              where we used (A) Remark \ref{REMARK-Cp-tail-hole}, (B) the fact that $w(x) = 0$ for every $x$ such that $\frac{\ell}{2} < |x| < 1$, and (C) for $|x| > 1$ we have $|J| \gtrsim |x|$ for every interval $J$ such that $\Omega \cap J \neq \emptyset$. Thus, we have
              $h \lesssim \ell^{p-1}$.

    \item[ii)] Suppose then that $h_k = (\ell_k)^{p-1}$. We want to show that there
               exist constants $C > 0$ and $\eps > 0$ such that they are independent of $I$ and $E$ and
              $$
                 w(E) \le C\Big( \frac{|E|}{|I|}\Big)^\eps \int_\R (M1_I)^p w.
            $$
               Naturally, we may assume that $w(I)>0$. We split the proof into two cases, depending on the interaction between $I$ and the support of $w$.

               \textbf{Case 1: $|I\cap A| > 0$}. By \eqref{prop:homogeneity}, we know that there exists a point $x_0 \in I\cap A$ such that
          $$
                 \big| A \cap ( x_0 - |I|, x_0 + |I|) \big| \ge c_A |I|.
             $$
               See Figure \ref{fig1} for this case.

               \begin{figure}[ht]
                 \centering
                 \begin{tikzpicture}[scale=1.25]


                    \draw[thin] (-3.5,0)-- (-0.5,0);
                    \draw[thin] (1.6,0)-- (7.5,0);

                    \node at (-3,0) {$\textbf{[}$};
                    \node at (-2,0) {$|$};
                    \node at (-2,-0.7) {$\underbrace{\hspace{2cm} }_{I_{k-1}}$};
                    \draw[ultra thick] (-3,0)--(-1,0);
                    \node at (-1,0) {$\textbf{]}$};

                    \node at (0,0) {$|$};

                    \node at (1,0) {$\textbf{[}$};
                    \node at (2,0) {$|$};
                    \node at (2,-0.7) {$\underbrace{\hspace{2cm} }_{I_{k}}$};
                    \draw[ultra thick] (1,0)--(3,0);
                    \node at (3,0) {$\textbf{]}$};

                    \node at (4,0) {$|$};

                    \node at (5,0) {$\textbf{[}$};
                    \node at (6,0) {$|$};
                    \node at (6,-0.7) {$\underbrace{\hspace{2cm} }_{I_{k+1}}$};
                    \draw[ultra thick] (5,0)--(7,0);
                    \node at (7,0) {$\textbf{]}$};

                    \node[red] at (-1/2,0)  {$\textbf{[}$};
                    \node[red] at (1.6,0)  {$\textbf{]}$};
                    \draw[color=red,ultra thick,dashed] (-1/2,0)--(1.6,0);
                    \node[red] at (0.55,0.7) {$\overbrace{\hspace{2cm} }^{I}$};


                    \node[blue] at (-0.8,0)  {$\Big($};

                    \filldraw[blue] (1.3,0) circle(1.5pt);
                    \small{\node[blue] at (1.3,0.2) {$x_0$};}
                    \normalsize
                    \node[blue] at (1.3,-1.2) {$\underbrace{\hspace{4.5cm} }_{(x_0-|I|,x_0+|I|)}$};
                    \node[blue] at (3.4,0)  {$\Big)$};
                 \end{tikzpicture}

                 \caption{Case 1: $|I\cap A|>0$.} \label{fig1}
               \end{figure}
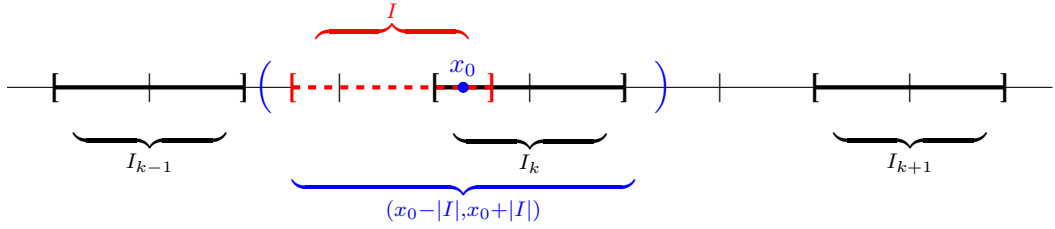
               \noindent Thus, since $1_A \le w \le 1$ a.e. and it holds that $(x_0 - |I|, x_0 + |I|)\subset 3I$, we have
               \begin{align*}
                 w(E)
                 \le |E|
                 &\le c_A^{-1} \frac{|E|}{|I|} \big| A \cap ( x_0 - |I|, x_0 + |I| ) \big| \\
                 &\le c_A^{-1} \frac{|E|}{|I|} w(3I) \le C_{A,p} \frac{|E|}{|I|} \int_\R (M1_I)^p w,
               \end{align*}
               where we used Lemma \ref{lemma:hl_bound} in the last inequality.

               \textbf{Case 2: $|I\cap A| = 0$}. In this case, we only have exactly one $k_0 \in \Z$ such that $I \cap \Omega_{k_0} \neq \emptyset$. Let us consider two subcases.

               \textbf{Case 2a: $|I|\le|\Omega_{k_0}|$}. In this case, we know that $w \le (\ell_{k_0})^{p-1}$ on $E \cap \Omega_{k_0}$. Thus, we get
               \begin{equation}
                 \label{estimate:case2a1}
                 w(E)
                 \le (\ell_{k_0})^{p-1} |\Omega_{k_0}\cap E|
                 \le (\ell_{k_0})^{p-1} |E|
                 = (\ell_{k_0})^{p-1} \frac{|E|}{|I|} |I|.
               \end{equation}
               Since $I \cap \Omega_{k_0} \neq \emptyset$ and $|I| \le |\Omega_{k_0}|$, there exists $\widetilde{I} \subset  \Omega_{k_0}$ such that $\widetilde{I} \subset 3I$ and $|\widetilde I| = |I|$. See Figure \ref{fig2a}.

               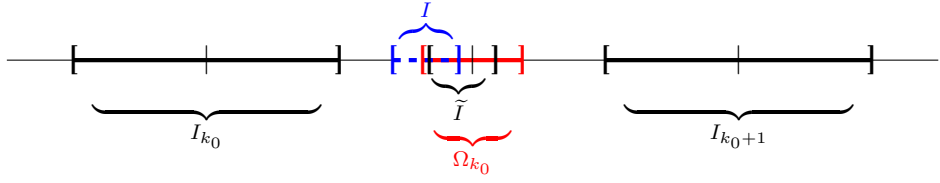
\begin{figure}[ht]
                 \centering
                 \begin{tikzpicture}[scale=1.75]

                    \draw[thin] (-3.5,0)-- (-0.6,0);
                    \draw[thin] (-0,0)-- (3.5,0);

                    \node at (-3,0) {$\textbf{[}$};
                    \draw[ultra thick] (-3,0)--(-1,0);
                    \node at (-2,0) {$|$};
                    \node at (-1,0) {$\textbf{]}$};
                    \node at (-2,-0.5) {$\underbrace{\hspace{3cm} }_{I_{k_0}}$};

                    \node[red] at (-3/8,0)  {$\textbf [$};
                    \node[red] at (3/8,0)  {$\textbf ]$};
                    \draw[red,ultra thick] (-3/8,0)--(3/8,0);
                    \node[red] at (0,-0.7) {$\underbrace{\hspace{1cm} }_{\Omega_{k_0}}$};

                    \node[blue] at (-0.6,0)  {$\textbf [$};
                    \node[blue] at (-0.35,0.3) {$\overbrace{ }^{I}$};
                    \draw[blue, ultra thick, dashed] (-0.6,0)--(-0.1,0);
                    \node[blue] at (-0.1,0)  {$\textbf ]$};

                    \node[black] at (-0.325,0)  {$\textbf [$};
                    \node[black] at (-0.105,-0.3) {$\underbrace{ }_{\widetilde I}$};
                    \node[black] at (0.175,0)  {$\textbf ]$};

                    \node at (0,0) {$|$};

                    \node at (1,0) {$\textbf{[}$};
                    \draw[ultra thick] (1,0)--(3,0);
                    \node at (2,0) {$|$};
                    \node at (3,0) {$\textbf{]}$};
                    \node at (2,-0.5) {$\underbrace{\hspace{3cm} }_{I_{k_0+1}}$};
                 \end{tikzpicture}

                 \caption{Case 2a: $|I\cap A|=0$, $|I|\leq |\Omega_{k_0}|$.} \label{fig2a}
               \end{figure}

               \noindent Thus, we have
               \begin{equation}
                 \label{estimate:case2a2}
                 \int_\R (M1_I)^p w
                 \ge \int_{\widetilde{I}} (M1_I)^p w
                 \gtrsim h_{k_0} |\widetilde{I}|
                 = (\ell_{k_0})^{p-1}|I|.
               \end{equation}
               Combining \eqref{estimate:case2a1} and \eqref{estimate:case2a2} then gives us
            $$
                 w(E) \leq C \frac{|E|}{|I|} \int_\R (M1_I)^p w,
          $$
               which is what we wanted.

               \textbf{Case 2b: $|I| > |\Omega_{k_0}|$}. In this case, we have obviously $\Omega_{k_0}\subset 3I$. See Figure \ref{fig2b}.

               \begin{figure}[h]
                 \centering
                 \begin{tikzpicture}[scale=1.75]
                    \draw[thin](-3.5,0)--(-0.5,0);
                    \draw[thin] (0.8,0)--(3.5,0);

                    \node at (-3,0) {$\textbf{[}$};
                    \draw[ultra thick] (-3,0)--(-1,0);
                    \node at (-2,0) {$|$};
                    \node at (-1,0) {$\textbf{]}$};
                    \node at (-2,-0.5) {$\underbrace{\hspace{3cm} }_{I_{k_0}}$};

                    \node[red] at (-3/8,0)  {$\textbf [$};
                    \node[red] at (3/8,0)  {$\textbf ]$};
                    \draw[red,ultra thick] (-3/8,0)--(3/8,0);
                    \node[red] at (0,-0.5) {$\underbrace{\hspace{1cm} }_{\Omega_{k_0}}$};

                    \node[blue] at (-1/2,0)  {$\textbf [$};
                    \node[blue] at (0.15,0.3) {$\overbrace{\hspace{1.95cm} }^{I}$};
                    \draw[blue, ultra thick, dashed] (-1/2,0)--(0.8,0);
                    \node[blue] at (0.8,0)  {$\textbf ]$};

                    \node at (0,0) {$|$};

                    \node at (1,0) {$\textbf{[}$};
                    \draw[ultra thick] (1,0)--(3,0);
                    \node at (2,0) {$|$};
                    \node at (3,0) {$\textbf{]}$};
                    \node at (2,-0.5) {$\underbrace{\hspace{3cm} }_{I_{k_0+1}}$};
                 \end{tikzpicture}

                 \caption{Case 2b: $|I \cap A| = 0$, $|I| > |\Omega_{k_0}|$.} \label{fig2b}
               \end{figure}
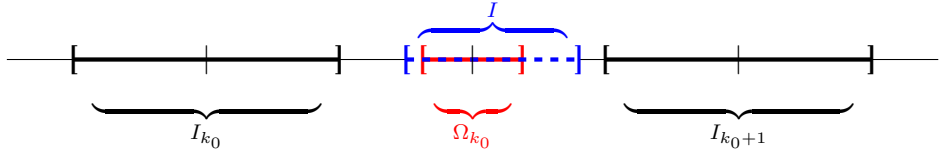
               \noindent Let $x_I$ be the center of $I$. Since $w \equiv 1$ on $I_{k_0+1}$ and $|I| \le |I_{k_0+1}| = 2$, we have
               \begin{align*}
                 \int_\R M(1_I)^p w
                 &\ge \int_\R \Big( \frac{|I|}{|I|+|x-x_I|}\Big)^p w \ge \int_{I_{k_0+1}}\Big(\frac{|I|}{|I|+|x-x_I|}\Big)^p w \\
                 &\ge c_p |I|^p \int_{I_{k_0+1}} \Big(\frac{1}{|I_{k_0+1}|}\Big)^p \ge C_p |I|^p.
               \end{align*}
               Since $\ell_{k_0} = |\Omega_{k_0}| < |I|$, we also have
             $$
                 w(E)
                 = (\ell_{k_0})^{p-1} |\Omega_{k_0} \cap E|
                 \le |\Omega_{k_0}|^{p-1}|E|
                 \le |I|^{p-1}|E|
                 = |I|^p \frac{|E|}{|I|}.
              $$
               Combining these two estimates gives us what we wanted. This completes the proof of part ii).

   \item[iii)]
 The proof of part iii) is similar to the proof of part i). We use the same
                notation as in the proof of part i). Using Lemma \ref{lemma-hole}, the facts that $\varphi_p$ is increasing and doubling and that $|\Omega| = \ell$, we get
                \begin{align*}
                  h \ell
                  = w(\Omega)
                  &\le C  \int_{\R} \varphi_p (M1_\Omega) w \le C  \int_{|x|>\ell } \varphi_p (M1_\Omega(x)) w(x) \, dx \\
                  &\le C  \int_{|x|>1} \varphi_p\Big( \frac{\ell }{|x|}\Big) \, dx  \lesssim C  \int_1^\infty \varphi_p\Big( \frac{\ell }{x}\Big) \, dx
                  = C  \ell \int_{0}^\ell \varphi_p(x) \, \frac{dx}{x^2},
                \end{align*}
                where we used integration by substitution in the last step.

   \item[iv)] We argue as in the proof of ii). The cases $1$ and 2a are essentially the same, since the value of $h_k$ does not really play a role in these cases. Let us prove the case 2b. We get
    \begin{align*}
      \int_\R \varphi_p(M1_I) w
      &\ge \int_\R \varphi_p \Big( \frac{|I|}{|I|+|x - x_I|}\Big) w \ge \int_{I_{k_0+1}}\varphi_p \Big( \frac{|I|}{|I|+|x-x_I|} \Big) w \\
      &\ge \int_{I_{k_0+1}}\varphi_p \Big( \frac{|I|}{2|I_{k_0+1}|}\Big) \ge \int_{I_{k_0+1}}\varphi_p \Big( \frac{|I|}{4}\Big)
      \ge C_\varphi \varphi_p(|I|),
    \end{align*}
    since $\varphi_p$ is increasing and doubling and $w = 1$ a.e. on $I_{k_0+1}$. Also, we have
$$
      w(E)
      = \frac{\varphi_p(\ell_{k_0})}{\ell_{k_0}} |\Omega_{k_0} \cap E|
      \le \frac{\varphi_p(|\Omega_{k_0}|)}{|\Omega_{k_0}|} |E|
      \overset{\text{(*)}}{\le} \frac{\varphi_p(|I|)}{|I|} |E|
      = \varphi_p(|I|) \frac{|E|}{|I|},
$$
    where we used the fact that $t \mapsto \frac{\varphi_p(t)}{t}$ is an increasing function in (*). This finishes the proof. \qedhere
 \end{enumerate}
\end{proof}

\subsection{Kahanp\"a\"a--Mejlbro weights in higher dimensions}

Although the definition of $\widetilde{C}_p$ makes sense in every dimension,
the proof of Theorem \ref{thm:strict_inclusions} works only in dimension one since it relies on the one-dimensional construction of Kahanp\"a\"a--Mejlbro weights and their properties.
In this section, we explain how the construction and the the proofs of Theorem \ref{theorem-counterexample} and Theorem \ref{thm:strict_inclusions} can be generalized for higher dimensions.

For a point $x = (x_1,\ldots,x_n) \in \R^n$ and $r > 0$, we let $Q(x,r)$ be the (closed) cube centered at $x$ with side length $2r$:
$$
  Q(x,r) \coloneqq [x_1 - r, x_1 + r] \times \ldots \times [x_n - r, x_n + r].
$$
Let us construct the $n$-dimensional analogue of the set $A$ from the proof of Theorem \ref{theorem-counterexample}. For every $m = (m_1, \ldots ,m_n) \in \Z^n$, we set
$$
  R_m \coloneqq Q(4m-2,1) = \left[ 4m_1-3, 4m_1-1 \right] \times \ldots \times \left[ 4m_n-3 ,4m_n-1 \right].
$$
We now use the cubes $R_m$ similarly as the intervals $I_k$ and set $A \coloneqq \bigcup_{m\in \Z^n} R_m$.

\begin{lemma}
  There exists a constant $c_A > 0$ such that, for every $x \in A$ and $r > 0$,
  \begin{equation}
    \label{prop:homogeneity_greater_dimensions}
    |A \cap Q(x,r)| \ge c_A r^n.
  \end{equation}

\end{lemma}

\begin{proof}
  Let us fix $x \in A$ and $0 < r < \infty$. Then $x$ lies inside exactly one of the cubes $R_m$. Let us denote this cube by $Q_0$.
  \begin{enumerate}
    \item[$\bu$] Suppose that $0 < r < 2$. Let us break $Q(x,r)$ into $2^n$ subcubes of
                 side length $r$. Since $x \in Q_0$ and $\ell(Q_0) = 2 > r$, at least one of the subcubes has to lie inside $Q_0$. Let us denote this subcube by $P$. Thus,
          $$
                   |A \cap Q(x,r)| \ge |Q_0 \cap Q(x,r)| \ge |P| = r^n.
          $$

    \item[$\bu$] Suppose that $2 + 4j \le r < 2+4(j+1)$ for some $j \ge 0$.
                 There are at least $(2j+1)^n$ cubes $R_m$ contained in $Q(x,r)$. Since each of these cubes has measure $2^n$, we get
                 \begin{align*}
                   |A \cap Q(x,r)|
                   & \ge (2j+1)^n 2^n
                   = \frac{(4j+2)^n}{r^n} r^n
                   \ge \frac{(4j+2)^n}{(2+4(j+1))^n} r^n
                   = \left( \frac{1}{3} \right)^n r^n.\qedhere
                 \end{align*}
  \end{enumerate}
\end{proof}
Let us then construct the $n$-dimensional weights. For every $m \in \Z^n$, let $\ell_m$ be a number such that $0 < \ell_m < 1$ and $\inf_m \ell_m = 0$. We set
\begin{align*}
  P_m \coloneqq Q\left(4m,\frac{\ell_m}{2}\right) = \left[4m_1-\frac {\ell_m}{2},4m_1+\frac {\ell_m }{2}\right] \times \ldots \times \left[4m_n-\frac {\ell_m }{2},4m_n+\frac {\ell_m }{2}\right],
\end{align*}
for every $m \in \Z$. Thus, we have $\ell(P_m) = \ell_m$. See Figure \ref{Figure-with-rectangles} for a visual description of the sets $A$ and $P_m$ in dimension $2$. Now we can define the Kahanp\"a\"a--Mejlbro weight $w$ in an obvious way as
\begin{align}\label{great-dimension-weight}
  w = 1_A + \sum_{m \in \Z^n} h_m 1_{P_m},
\end{align}
where $h_m$ are numbers such that $0 < h_m < N$ for every $m \in \Z^n$ for a uniformly bounded constant $N$. Naturally, these weights share a lot of properties with their
$1$-dimensional counterparts but because of the dimension, we have to make some modifications.

\begin{center}
  \begin{figure}
    \begin{tikzpicture}[scale=0.7]

        \draw[] (-9.3,-9)--(9,-9);
        \foreach \x in {-8,...,8}
            \draw (\x,-9.1) -- (\x,-8.9);
        \foreach \x in {-8,-4,0,4,8}
            \node[anchor=north] at (\x,-9.1) {\x};	

        \draw (-9,-9.3)--(-9,9);
        \foreach \y in {-8,...,8}
            \draw (-9.1,\y) -- (-8.9,\y);
        \foreach \y in {-8,-4,0,4,8}
            \node[anchor=east] at (-9.1,\y) {\y};

                \foreach \x in {-1,0,1,2}
                \foreach \y in {-1,0,1,2}
                        \draw[draw=red,fill=red!20]  (4*\x-3,4*\y-3) rectangle (4*\x-1,4*\y-1);

                \foreach \x in {-1,0,1,2}
                \foreach \y in {-1,0,1,2}
                        \node at  (4*\x-2,4*\y-2) {$R_{(\x,\y)}$};



                \foreach \x in {-2,-1,0,1,2}
                \foreach \y in {-2,-1,0,1,2}
                    \draw[draw=blue,fill=blue!20] ({4*\x-0.5/(abs(\x)+abs(\y)+1)},{4*\y-0.5/(abs(\x)+abs(\y)+1)}) rectangle ({4*\x+0.5/(abs(\x)+abs(\y)+1)},{4*\y+0.5/(abs(\x)+abs(\y)+1)});


            \node at (0,0) {\scriptsize{$ P_{(0,0)}$}};

    \end{tikzpicture}    		
    \caption{The cubes $R_m$ (in red) and $P_m$ (in blue) in $\R^2$, with $m=(m_1,m_2)$, for
             $\ell_m = \frac{1}{|m|+1}$. Each $R_m$ has side length $2$ and $P_m$ has sidelength $\ell_m$.}
    \label{Figure-with-rectangles}
  \end{figure}
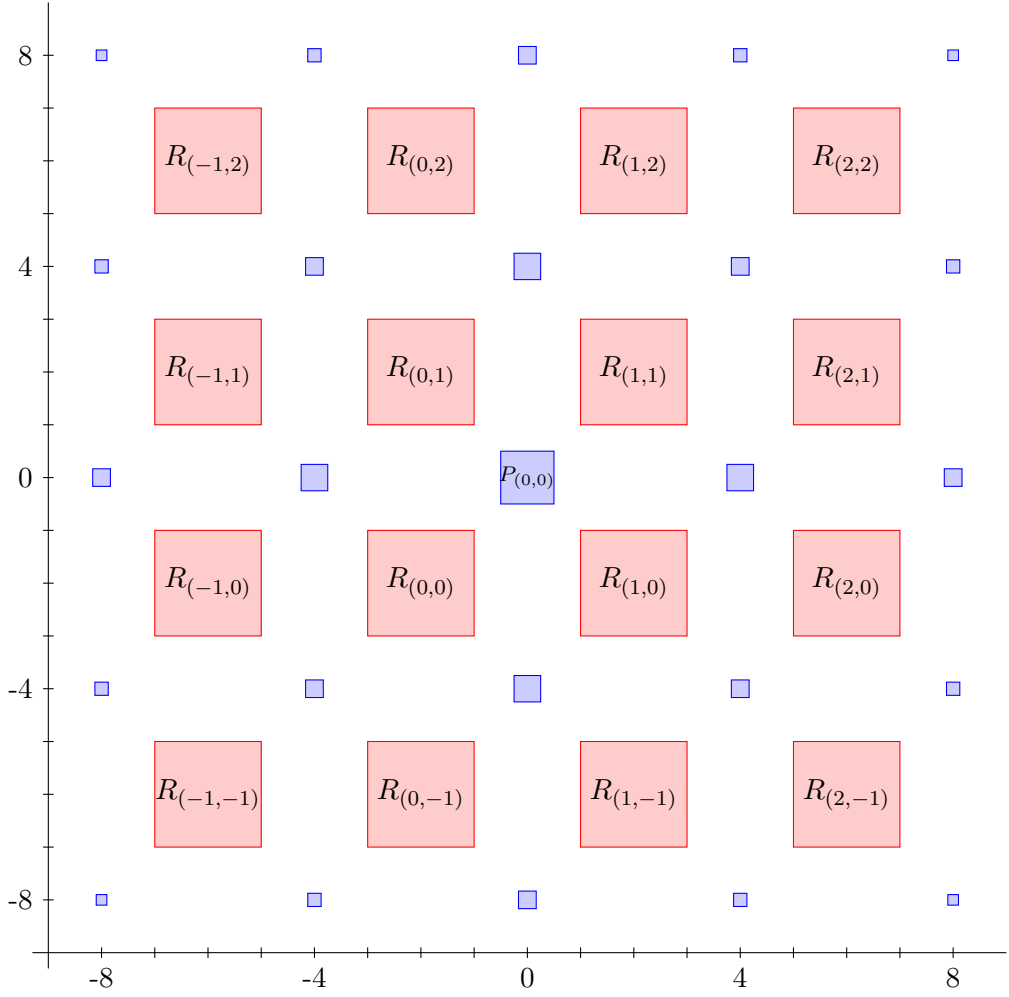
\end{center}

An analogue of Theorem \ref{theorem-counterexample} holds for these $n$-dimensional weights in the following form:
\begin{theorem}
  \label{great-dimension.theorem}
  Let $w$ a weight as in \eqref{great-dimension-weight}.
  \begin{enumerate}
    \item[i)] If $w \in C_p$, then $h_m \lesssim (\ell_m)^{n(p-1)}$.
    \item[ii)] If $h_m = (\ell_m)^{n(p-1)}$, then $w \in C_p$.
    \item[iii)] If $w \in \widetilde{C}_p$, then $h_m \lesssim \int_0^{(\ell_m)^n} \varphi_p(t) \, \frac{dt}{t^2}$.
    \item[iv)] If $h_m = \frac{\varphi_p(\ell_m^n)}{\ell_m^n}$, then $w \in \widetilde{C}_p$.
  \end{enumerate}
\end{theorem}
The correct exponent is now $n(p-1)$ instead of $p-1$ because $|P_m|=(\ell_m)^n$.

The proof of this theorem is essentially the same as in the $1$-dimensional case. Since Lemma~\ref{lemma-hole} holds in any dimension, the proofs of i) and iii) work also in any dimension. Parts ii) and iv) also hold because of \eqref{prop:homogeneity_greater_dimensions} and there are no more cases than the $1$-dimensional cases 1, 2a and 2b. The rest of the computations are essentially the same as before.

With the help of Theorem \ref{great-dimension.theorem}, it is straightforward to generalize Theorem \ref{thm:strict_inclusions} for higher dimensions:

\begin{theorem}
\label{thm:anyd}
  In any dimension, we have
  \begin{enumerate}
    \item[a)] $C_p \setminus \widetilde C_p \neq \emptyset$,
    \item[b)] $ \cup_{\varepsilon>0} C_{p+\varepsilon} \subsetneq \widetilde C_p$.
  \end{enumerate}
  In particular, the condition $C_{p+\eps}$ is not necessary for \eqref{ineq:coifman-fefferman} to hold for Calder\'on--Zygmund operators.
\end{theorem}

\section*{Acknowledgments}
The active work on this manuscript began in May 2019 when the fourth author visited the other authors at the Basque Center for Applied Mathematics in Bilbao. He wishes to thank the people at BCAM for numerous interesting conversations and the kind hospitality shown during his visit.

\end{document}